\documentclass[onefignum,onetabnum]{siamart251216}

\usepackage{rotating,array,amsmath}
\usepackage{amssymb}
\usepackage{bbm}	
\usepackage{bm}		
\usepackage{mathtools} 

\usepackage{array}    
\usepackage{booktabs} 
\usepackage{makecell} 
\usepackage{url}
\usepackage{xcolor} 
\usepackage[ruled,vlined, algo2e]{algorithm2e}
\usepackage{stmaryrd}
\usepackage{mathrsfs}
\usepackage{caption}
\usepackage{graphicx}
\usepackage{subcaption} 
\usepackage{hyperref}

\newtheorem{assumption}{Assumption}
\newtheorem{notation}{Notation} 
\newtheorem{remark}{Remark}

\newcommand{\R}{\mathbb{R}}
\newcommand{\N}{\mathbb{N}}
\newcommand{\IntInterval}[2]{\llbracket #1, #2 \rrbracket} 
\newcommand{\En}{\mathcal{E}}

\newcommand{\python}{{\ttfamily Python}}
\newcommand{\mosek}{{\ttfamily mosek}}
\newcommand{\celer}{{\ttfamily celer}}
\newcommand{\cpp}{{\ttfamily C++}}

\DeclarePairedDelimiter{\norm}{\lVert}{\rVert}	
\newcommand{\normad}[1]{\norm*{#1}}				

\newcommand{\argmax}{\mathop{\mathrm{argmax}}}

\newcommand{\diag}[1]{\mathop{\mathrm{diag}}\left[#1\right]}

\newcommand{\indic}{\mathbbm{1}_{n}} 
\newcommand{\indicm}{\mathbbm{1}}
\DeclareMathOperator{\KL}{\mathsf{KL}} 

\newcommand{\dmax}{\delta_\mathrm{max}}
\newcommand{\tree}{\mathscr{T}}
\newcommand{\leftt}{\mathbf{s}}
\newcommand{\rightt}{\mathbf{t}}
\newcommand{\roott}{\mathbf{r}}
\newcommand{\children}{\mathbf{c}}
\newcommand{\parent}{\mathbf{p}}

\newcommand{\mmu}{\bm{\mu}}
\newcommand{\nnu}{\bm{\nu}}
\newcommand{\ff}{\mathbf{f}}
\renewcommand{\gg}{\mathbf{g}}
\newcommand{\ww}{\mathbf{w}}
\newcommand{\aalpha}{\bm{\alpha}}
\newcommand{\oone}{\bm{1}}
\newcommand{\ttwo}{\bm{2}}
\newcommand{\ZZ}{\mathbb{Z}}

\renewcommand{\succeq}{\succcurlyeq}
\renewcommand{\preceq}{\preccurlyeq}

\newcommand{\margset}{\mathbf{U}(\mmu,\nnu)}
\newcommand{\costmat}{\mathbf{C}}
\newcommand{\tpplan}{\mathbf{P}}

\title{A Network Simplex Framework for Discrete Unbalanced Optimal Transport}
\author{
Gaetano Agazzotti
\thanks{CEREMADE, CNRS, Université Paris Dauphine-PSL, France 
and Mokaplan, INRIA Paris
({\tt agazzotti@ceremade.dauphine.fr}).}
\and
Antonin Chambolle
\thanks{CEREMADE, CNRS, Université Paris Dauphine-PSL,  France 
and Mokaplan, INRIA Paris, 
({\tt chambolle@ceremade.dauphine.fr}).}
\and
Clément W. Royer
\thanks{Univ Gustave Eiffel, CNRS, LIGM, F-77454 Marne-la-Vall\'{e}e, France
({\tt clement.royer@univ-eiffel.fr}).
Funding for this research was provided by 
\emph{Agence Nationale de la Recherche} through a ``France 2030'' support 
ANR-23-PEIA-0004 (PEPR PDE-AI), and through program ANR-23-IACL-0008 
(PR[AI]RIE-PSAI).}
}
\date{\today}

% ...

\begin{document}

\maketitle

\begin{abstract}
    The network simplex method is a classical algorithm for minimum cost flow 
    problems formulated as linear programs, that has attracted recent 
    interest from the optimal transport community. Indeed, discrete optimal 
    transport is a subclass of minimum cost flow for which network simplex 
    techniques compute exact solutions while leveraging the inherent problem 
    sparsity. However, extending these approaches to other, nonlinear optimal 
    transport formulations has yet to be thoroughly investigated.
    In this paper, we adapt the network simplex algorithm to unbalanced 
    optimal transport.
    Our proposed framework is applicable to various 
    unbalanced formulations, and can be implemented explicitly in the most 
    common cases, including semi-unbalanced optimal transport.
    For quadratic relaxations, we further establish worst-case 
    complexity bounds for the resulting algorithms.
    In practice, these algorithms are memory-efficient, and perform 
    competitively with existing approaches, as shown by our experiments on a 
    range of transport and barycenter calculations.
\end{abstract}

\section{Introduction}

Optimal transport~\cite{santambrogio2015optimal,villani2008optimal} (OT) is a 
rich mathematical model that has drawn significant attention in the last decades, 
fueled by a number of applications in computer vision~\cite{bonneel2023survey} 
and data science~\cite{GPeyre_MCuturi_2019}. An optimal transport problem 
consists in computing a \textit{transport plan} to move mass from one measure 
to another so as to minimize a transport cost per unit of transported mass. 
In its simplest form, this problem is a linear program (possibly of infinite 
dimension), and can thus be tackled by existing linear programming 
techniques~\cite{GPeyre_MCuturi_2019}. In the discrete setting, OT belongs 
to a class of linear programs called minimum-cost flow problems, for which a 
number of algorithmic approaches have been proposed~\cite{nf1993orlin}. The
network simplex emerges as the best method for this class of problems in 
practice, despite other variants possessing better theoretical complexity 
guarantees~\cite{KCornelissen_BManthey_2020,PKovacs_2015}. 
Implementations of the network simplex used in optimal transport are also 
among the fastest schemes to solve discrete optimal transport 
problems~\cite{NBonneel_MvandePanne_SParis_WHeidrich_2011,pot}. Moreover, 
compared to first-order schemes such as Sinkhorn's 
algorithm~\cite{SinkhornKnopp}, that rely on entropic regularization, the 
network simplex method can exhibit sparser transport plans while solving the 
original problem exactly. We note that recent work has applied interior-point 
methods to optimal 
transport~\cite{SCipolla_JGondzio_FZanetti_2024,FZanetti_JGondzio_2023}, 
but that those methods only compute approximate solutions and require 
significant adaptation in order to leverage the (graphical) problem structure.

Standard OT formulations assume that the two measures involved have identical 
mass, in order for the problem to have a finite optimal value. This assumption 
does not always hold in practice. For instance, color transfer formulations do 
not necessarily assume identical mass between the source and target color 
distributions~\cite{MBlondel_VSeguy_ARolet_2018}. To circumvent this issue, 
a generalized version of optimal transport termed \textit{unbalanced optimal 
transport} (UOT) was introduced by Benamou~\cite{benamou2003numerical}. The 
UOT formulation removes marginal constraints and replaces them by penalty 
terms (also called relaxations) in the objective. Semi-unbalanced formulations 
(SUOT), that are obtained by relaxing only one of the marginal constraints, 
arise in Wasserstein gradient flow calculations~\cite{JKO1998,peyre2015entropic} 
as well as image restoration~\cite{galerne2025}.

As in classical optimal transport, Sinkhorn-type algorithms can be used to tackle 
UOT problems~\cite{LChizat_GPeyre_BSchmitzer_FXVialard_2018,sejourne2019sinkhorn}. 
However, as in OT, this approach outputs approximate solutions and typically 
dense transport plans. Other proposals have thus been developed for quadratic and 
Kullback-Leibler UOT formulations~\cite{chapel2021unbalanced,
MEssid_JSolomon_2018}, including the so-called regularization path method 
that yields an exact solution to the UOT problem in the quadratic 
case~\cite{chapel2021unbalanced}.
Note that UOT with convex relaxation terms is a minimum cost flow problem
with convex costs, a topic of rich history in the network optimization 
community~\cite{DPBertsekas_1998}. Several algorithms termed ``relaxation 
methods'' have been proposed to tackle such flow problems with theoretical 
guarantees~\cite{vegh2016} and practical 
performance~\cite{bertsekas1997relaxation,deleone1999partitioned}. To the best 
of our knowledge, however, implementations of these algorithms are not readily 
available. In addition, these techniques are outperformed by network simplex 
variants such as that used in optimal transport packages in the linear 
case~\cite{PKovacs_2015}. As a result, one may wonder whether network 
simplex techniques can be extended to the UOT setting.

In this paper, we propose an algorithmic framework to solve unbalanced and 
semi-unbalanced transport problems derived from the standard network simplex 
algorithm. Our framework leverages the bipartite graph structure inherent to 
optimal transport for updating primal and dual variables in a 
memory-efficient way. In addition, iterates are explicitly defined for 
standard formulations of unbalanced transport based on quadratic and 
Kullback-Leibler divergences. We provide a complexity analysis for our method 
in the former case, comparing it with the best known theoretical guarantees. 
In our numerical experiments, the algorithm competes with mature solvers 
dedicated to quadratic and Kullback-Leibler examples. We also use our method 
for barycenter calculations, where unbalanced optimal transport arises as a 
subtask, and find again that our method performs comparably to other schemes.

The rest of this paper is organized as follows. Section~\ref{sec:uot} 
describes the unbalanced optimal transport paradigm. Section~\ref{sec:algo} 
describes our main algorithm and its connection to the network simplex method. 
Section~\ref{sec:particular-cases} derives explicit formulas for our
algorithm's iterates when quadratic and Kullback-Leibler relaxations are 
used. Section~\ref{sec:complexity} provides complexity guarantees for our 
algorithm. Section~\ref{sec:num-exp} showcases the practical benefits of our 
approach compared to existing implementations on unbalanced and 
semi-unbalanced problems. Section~\ref{sec:conc} summarizes the contents of 
the paper.

\paragraph{Notation} Throughout this paper, vectors are denoted with bold 
lowercase letters ($\bm{a}$) while matrices are denoted by uppercase letters 
($\mathbf{A}$). Given two matrices $\mathbf{A}$ and 
$\mathbf{B}$, $\mathbf{A}\!:\!\mathbf{B}$ denotes the Frobenius inner product. 
Given two vectors $\bm{u} \in \R^{n_1}$ and $\bm{v} \in \R^{n_2}$, we define 
$\bm{u} \oplus \bm{v} := [\bm{u}_i+\bm{v}_j]_{ij} \in \R^{n_1 \times n_2}$. 
Given two  vectors (or matrices) $\bm{a}$ and $\bm{b}$ of same dimension(s), 
we let $\bm{a} \odot \bm{b}$ denote their componentwise product, and 
$\bm{a} \oslash \bm{b}$ denote their componentwise division (assuming 
$\bm{b}>0$ in the latter case). 
The convex conjugate of a convex function $\psi: \R^m \rightarrow \R$ will be 
denoted by $\psi^*$, where $\psi^*(\bm{x})=
\max_{\bm{y}\in\R^{m}} \langle \bm{x},\bm{y}\rangle - \psi(\bm{y})$ for 
every $\bm{x} \in \R^m$. When considering bipartite graphs of the form 
$(V \cup V',E)$, we will write non-oriented edges as $(i,j)$, where 
$i$ denotes an index in the first part of the graph and $j$ denotes a vertex 
from its second part.

\section{Discrete unbalanced optimal transport}
\label{sec:uot}

In this section, we provide background material on both classical and 
unbalanced optimal transport. Section~\ref{subsec:ot} introduces the optimal 
transport problem and explains its connection with minimum cost flow, 
while Section~\ref{subsec:uot} focuses on unbalanced optimal transport.

\subsection{Discrete optimal transport}
\label{subsec:ot}

We first define the discrete optimal transport problem and its connections 
with minimum cost flows, following the presentation of Peyr\'e and 
Cuturi~\cite[Sec. 3.5]{GPeyre_MCuturi_2019}.
Let $\mmu \in \R_+^{n_1}$ and $\nnu \in \R_+^{n_2}$ be two measures 
with the same mass (i.e. $\|\mmu\|_1=\|\nnu\|_1$). Given a nonnegative cost 
matrix $\costmat \in \R_+^{n_1 \times n_2}$, an optimal transport plan 
between $\mmu$ and $\nnu$ can be computed by solving 
\begin{equation}
\label{eq:ot}
    \min_{\tpplan\in \margset} \tpplan\!:\! \costmat 
    \quad \mbox{where} \quad
    \margset :=\{
    	\tpplan\in\R_{+}^{n_1\times n_2}
    	|
    	\tpplan\indicm_{n_2} = \mmu, 
    	\tpplan^\top\indicm_{n_1} = \nnu 
    	\}.
\end{equation}
A transport plan $\tpplan$ is a solution of~\eqref{eq:ot} if and 
only if there exist $\ff \in \R^{n_1}$ and $\gg \in \R^{n_2}$ such that 
\begin{equation}
\label{eq:opticond}
	\tpplan \in \margset, 
	\quad
	\ff \oplus \gg \le \costmat, 
    \quad
	(\costmat - \ff\oplus\gg)\odot\tpplan = \bm{0}_{n_1n_2} .
\end{equation}
The optimality conditions~\eqref{eq:opticond} correspond to primal 
feasibility, dual feasibility and complementarity, in that order. The latter 
condition implies that the optimal transport plan is easy to determine 
given a fixed support for $\tpplan$ and dual variable values. In fact, if 
the correct support is known, a solution can be computed from the dual 
variables, and the network simplex method is partly based on this observation. 
Note that problem~\eqref{eq:ot} can be equivalently reformulated as
\begin{equation}
\label{eq:otpd}
    \min_{\tpplan\geq 0}\ 
    \max_{(\ff,\gg)\in\R^{n_1} \times \R^{n_2} }
    \tpplan\!:\! (\costmat-\ff\oplus\gg) + 
    \langle \ff,\mmu\rangle + \langle \gg,\nnu\rangle.
\end{equation}
Problem~\eqref{eq:ot} can be cast as a minimum cost flow 
problem~\cite{nf1993orlin} on a bipartite graph. Indeed, let 
$V=\IntInterval{1}{n_1}$ and $V'=\IntInterval{1}{n_2}$ be source nodes 
and target nodes representing the coordinates of $\mmu$ and $\nnu$, 
respectively, and let $\mathcal{E} = V\times V'$ be a set of edges so 
that $(V \cup V',E)$ is a bipartite graph. 
Problem \eqref{eq:ot} then amounts to sending mass from source nodes to 
target nodes along edges while respecting the marginal constraints and 
minimizing the transport cost. Thus, any transport plan $\tpplan$
corresponds to a subgraph of $(V \cup V', \mathcal{E})$ with edge $(i,j)$
included if and only if $\tpplan_{i,j}>0$. In particular, vertices 
of $\margset$ correspond to subgraphs with no cycles (i.e. trees or 
forests), which include at most $n_1+n_2-1$ 
edges~\cite[Thm. 8.1.2, Cor. 8.1.3]{brualdi2006combinatorial}. It follows that 
there exists a solution of problem~\eqref{eq:ot} with at most $n_1+n_2-1$
nonzero coefficients. 

\subsection{Unbalanced optimal transport}
\label{subsec:uot}

We now describe unbalanced optimal transport formulations built from the 
optimal transport problem~\eqref{eq:ot} by relaxing the marginal constraints. 
Let $\psi_1:\R^{n_1}\to\R$ and $\psi_2:\R^{n_2}\to\R$ be two separable, 
strictly convex, and continuously differentiable functions that possess minima.
The unbalanced optimal transport problem associated with~\eqref{eq:ot} and 
$(\psi_1,\psi_2)$ is
\begin{equation}
\label{eq:rot}
    \min_{(\mmu,\nnu) \in \R^{n_1+n_2}
    }
    \min_{\tpplan\in\margset}
    \En(\tpplan) :=
    \tpplan\!:\!\costmat 
    + \psi_1(\mmu)
    + \psi_2 (\nnu).
\end{equation}
Let $\En^*$denote the minimal energy of \eqref{eq:rot}. Problem~\eqref{eq:rot} 
is an unbalanced version of~\eqref{eq:ot} in that the measures $\mmu$ and $\nnu$
need not have equal mass. By eliminating the measures 
$\mmu$ and $\nnu$, problem~\eqref{eq:rot} can be rewritten as
\begin{equation}
\label{eq:uot}
	\min_{\tpplan\geq 0}
    \tpplan\!:\!\costmat 
    + \psi_1(\tpplan\indicm_{n_2})
    + \psi_2 (\tpplan^\top\indicm_{n_1}),
\end{equation}
or, alternatively, as the primal-dual problem
\begin{equation}
\label{eq:rot-pd}
    \min_{\tpplan\geq 0}
    \max_{(\ff,\gg)\in\R^{n_1+n_2} }
    \tpplan\!:\! (\costmat-\ff\oplus \gg)
    - \psi_1^*(-\ff)
    - \psi_2^*(-\gg).
\end{equation} 
Note that the convex conjugates $\psi_1^*$ and $\psi_2^*$ are 
strictly convex and continuously differentiable. 
The optimality conditions for problem~\eqref{eq:uot} differ from those 
of~\eqref{eq:opticond} in the primal and dual feasibility conditions. 
The former is simply $\tpplan \ge 0$, while the latter become
\begin{equation}
\label{eq:dualfeasuot}
    \begin{cases}
        \tpplan\indicm_{n_2} = \nabla \psi_1^*(-\ff),\\
        \tpplan^\top\indicm_{n_1} = \nabla \psi_2^*(-\gg),
    \end{cases}
    \quad \Leftrightarrow \quad
    \begin{cases}
        \ff = -\nabla\psi_1(\tpplan\indicm_{n_2})\,,\\
        \gg = -\nabla\psi_2(\tpplan^\top\indicm_{n_1}).
    \end{cases}
\end{equation}
In addition, the complementarity condition is identical to that 
of~\eqref{eq:ot}, i.e.
\begin{equation}
\label{eq:compuot}
	\forall (i,j) \in \IntInterval{1}{n_1} \times \IntInterval{1}{n_2}, 
	\quad
	(\costmat_{i,j} -\ff_i -\gg_j) \tpplan_{i,j} = 0.
\end{equation}
Similarly to OT, we can compute a primal solution for the problem given 
dual variables and the correct support. Indeed, the marginals 
$\tpplan\indicm_{n_2}$ and $\tpplan^\top\indicm_{n_1}$ are obtained from 
the dual variables~\eqref{eq:dualfeasuot}, while complementarity ensures 
that the primal variable $\tpplan$ solves the optimal transport problem.
In Section \ref{sec:algo}, we leverage this observation to adapt network simplex 
techniques to the unbalanced setting.

Beforehand, we state and prove two auxiliary results on the primal variable 
$\tpplan$ of \eqref{eq:rot}. Proposition \ref{lem:l1-norm-solutions} bounds its 
mass, and will be used on iterates of our network simplex technique. 
Proposition~\ref{prop:delta-precision-gen} then bounds the
optimality gap of the primal energy in terms of the dual violation.

\begin{proposition}
    \label{lem:l1-norm-solutions}
    For any triplet $(\tpplan,\ff,\gg)$ satisfying primal feasibility 
    $\tpplan \ge 0$, complementary condition \eqref{eq:compuot} and 
    first order optimality condition \eqref{eq:dualfeasuot}, we have:
    \begin{equation*}
        \lVert\tpplan\rVert_1
        \leq
            \lVert {\bm m}^{\bm 1}\rVert_1+
            \lVert {\bm m}^{\bm 2}\rVert_1
    \end{equation*}
    where ${\bm m}^{\bm k} := \mathrm{argmin}_{{\bm x}\in\R_+^{n_k}} 
    \psi^{\bm k}(\bm x)$, $k=1,2$.
    In particular, the result holds for 
     any solution of \eqref{eq:rot}.
\end{proposition}

\begin{proof}
    If ${\bm m}_i^{\bm k}>0$, $(\psi_i^{\bm k})'({\bm m}_i^{\bm k}) = 0$
    and $(\psi_i^{\bm k})'({\bm m}_i^{\bm k}) \ge 0$ if ${\bm m}_i^{\bm k} = 0$.
    Since $(\psi_i^{\bm k})'$ is strictly increasing, we have:
    \begin{equation}
        \ff_i\ge 0 \implies (\tpplan\indicm_{n_2})_i \le {\bm m}_i^{\bm 1}
        ~\text{and}~
    \gg_j\ge 0 \implies (\tpplan^\top\indicm_{n_1})_j \le {\bm m}_j^{\bm 2}. 
    \end{equation} 
    Now, consider $U = \{i\in V|\ff_i\ge 0\}$ and $U' = \{j\in V'|\gg_j\ge 0\}$. 
    For an arc $(i,j)\in\mathrm{supp}(\tpplan)$, we have $\ff_i\ge 0$ and/or $\gg_j\ge0$
    by complementarity, and thus $i\in U$ and/or $j\in U'$.
    It follows that
    \begin{equation}
        \lVert\tpplan\rVert_1
        = \sum_{(i,j)\in\mathrm{supp}(\tpplan)} \tpplan_{ij}
        \le 
        \sum_{i\in U} (\tpplan\indicm_{n_2})_i
        +
        \sum_{j\in U'} (\tpplan^\top\indicm_{n_1})_j
        \le 
        \sum_{i\in U} {\bm m}_i^{\bm 1}
        +
        \sum_{j\in U'} {\bm m}_j^{\bm 2}
    \end{equation}
    and the desired result holds.
\end{proof}

\begin{proposition}
    \label{prop:delta-precision}
    \label{prop:delta-precision-gen}
    Suppose that $(\tpplan,\ff,\gg)$ 
    satisfies $\tpplan\geq 0$, \eqref{eq:compuot}, 
    \eqref{eq:dualfeasuot} and 
    $\max[\ff\oplus \gg-\costmat] = \Delta\geq 0$.
    Then, we have $\En(\tpplan)-\En^*\le \Delta\lVert\tpplan^*\rVert_1$.
\end{proposition}

\begin{proof}
	By definition, the variables $(\tpplan,\ff,\gg)$, 
	are optimal for a UOT problem where the cost matrix in~\eqref{eq:uot} 
	is replaced by $\costmat+\Delta\,\imath_{\mathbf{P} = 0}$, where 
	$\imath_{\mathbf{P} = 0} \in \R^{n_1 \times n_2}$ is given by
	$(\imath_{\mathbf{P} = 0})_{i,j} = 1$ if $\mathbf{P}_{i,j}=0$ 
	and $0$ otherwise. In particular, optimality of $\mathbf{P}$ on the 
	modified problem implies
    \begin{equation*}
    	\En(\tpplan) 
    	+ \tpplan\!:\! 
    	\left(\Delta\cdot\imath_{\mathbf{P} = 0}\right)
    	\; \le \;
    	\En(\tpplan^*) 
    	+ \tpplan^*\!:\!
    	\left(\Delta\,\imath_{\mathbf{P} = 0}\right).
    \end{equation*}
    Since $\tpplan\!:\!\imath_{\mathbf{P} = 0} =0$ by 
    construction, it follows that 
    \[
    	\En(\tpplan)-\En(\tpplan^*)
    	\le \Delta \cdot\tpplan^*\!:\!\imath_{\mathbf{P} = 0}
        \le \Delta \lVert \tpplan^* \rVert_1
    \]
    and the desired conclusion holds.
\end{proof}

\section{Network simplex algorithms}
\label{sec:algo}

This section generalizes the network simplex framework for optimal transport 
to the unbalanced setting. We first review the key principles of the classical 
network simplex for optimal transport in Section~\ref{subsec:nsot}. We then 
adapt a key ingredient of network simplex (flow push) in 
Section~\ref{subsec:feasflowuot}. Our adapted network simplex is stated in 
Section~\ref{subsec:nsuot}.

\subsection{Network simplex for OT}
\label{subsec:nsot}

Network simplex methods~\cite{nf1993orlin} are variants of the simplex method 
for linear programming that exploit graph features of network problems such 
as~\eqref{eq:ot}. As in a standard simplex method, the goal consists in 
exploring the vertices of the feasible set to search for an exact solution. 

A primal network simplex method applied to~\eqref{eq:ot} begins with a vertex 
$\tpplan$ of $\margset$, typically computed through the so-called 
``North-West'' corner rule~\cite[Sec. 3.4.2]{GPeyre_MCuturi_2019}. Dual or 
complementary variables $(\ff,\gg)$ are then computed so that 
$(\tpplan,\ff,\gg)$ satisfies both primal feasibility and complementarity 
conditions in~\eqref{eq:opticond}. If dual feasibility is also satisfied, the 
triplet is an optimal solution. Otherwise, there must exist a \emph{pivot} 
edge $(i_0,j_0)\in\mathcal{E}$ such that 
$\ff_{i_0}+\gg_{j_0} > \costmat_{i_0,j_0}$. In that case, the edge $(i_0,j_0)$ 
is added to the subgraph associated with $\tpplan$. If this procedure does not 
create a cycle, i.e. if it links two disjoint trees within the support forest, 
then only the dual variables needs to be updating~\cite[p.412]{GPeyre_MCuturi_2019}. 
Otherwise, the cycle is eliminated by \emph{pushing a flow} along its edges. 
More precisely, let 
$\mathcal{C}:=\{(i_0,j_0),(j_0,i_1),(i_1,j_1),...,
(j_{\ell},i_\ell=i_0)\}$ 
denote the cycle of interest (recall that the vertices 
are identified using the bipartite structure). To push a flow along the cycle, 
we first compute $\theta_\mathcal{C} = 
\min_{(i_{k+1},j_k)\in\mathcal{C}} \tpplan_{i_{k+1},j_k}$, that represents how much 
flow one can push along $\mathcal{C}$ while maintaining primal feasibility 
($\tpplan \ge 0$). Given this value, we then update the transport plan values 
along the cycle as follows:
\begin{equation}
\label{eq:push-cycle}
    \tpplan_{i_0,j_0} = \theta_{\mathcal{C}} 
    \qquad \text{and} \qquad
    \forall k\in\IntInterval{1}{\ell-1}, 
    \quad
    \begin{cases}
        \tpplan_{i_k,j_k} 
        \gets \tpplan_{i_k,j_k} + \theta_\mathcal{C} \\
        \tpplan_{i_{k+1},j_k} 
        \gets \tpplan_{i_{k+1},j_k} - \theta_\mathcal{C}. 
    \end{cases}
\end{equation}
By construction, the coefficient  of
at least one (backward) edge $(i_{k+1},j_k)$ in the cycle is cancelled in 
the new $\tpplan$.
This edge (or the first one, in case there are several)
is deleted from the cycle, to produce a new tree.
This procedure 
produces a new vertex of $\margset$ while provably reducing the primal cost 
when $\theta_{\mathcal{C}}>0$. The case $\theta_{\mathcal{C}}=0$ is handled by the strategy of
removing the first edge corresponding to a zero coefficient of $\tpplan$, so 
as to guarantee that the method does not cycle and the overall process 
terminates in finite time~\cite[Sec. 3.5.3]{GPeyre_MCuturi_2019}. Depending 
on the pivot strategy, additional guarantees can be provided~\cite{nf1993orlin}, 
as we explain in Section~\ref{sec:complexity}. 

Efficient practical implementations of the network simplex method pay special 
attention to storage and data structures. In particular, the transport plan 
should be stored in sparse format along with the $n_1+n_2$ dual variables, and 
updates should only be performed on cycles and/or trees affected by a flow 
push~\cite{NBonneel_MvandePanne_SParis_WHeidrich_2011}. To the best of our 
knowledge, the fastest available implementation relies on the \texttt{LEMON} 
\cpp~library~\cite{lemon} and is adapted to OT in the \texttt{pot} 
\python~library~\cite{pot}.

\subsection{Pushing flows in unbalanced optimal transport}
\label{subsec:feasflowuot}

We now generalize the various ingredients of the network simplex method to 
tackle the unbalanced transport problem~\eqref{eq:uot}.
A first obvious, yet important remark is that
an optimum for~\eqref{eq:uot} is also an optimal
transport plan for its own marginals $\tpplan\oone_{n_2}$
and $\tpplan^\top\oone_{n_1}$, so that~\eqref{eq:rot} admits a solution
whose support is a tree or a forest.
Therefore, as in the original 
network simplex method, we iterate over transport plans associated with trees 
or forests in the underlying bipartite graph. Updating those plans involves 
checking conditions~\eqref{eq:dualfeasuot}, where both the transport plan 
$\tpplan$ and the complementary variables $(\ff,\gg)$ appear, and thus the 
transport plan cannot be updated independently of the dual variables in the 
unbalanced case. The transport plan can, however, be updated recursively 
by leveraging the problem's graph structure, and the following definitions.

\begin{notation} 
For any tree subgraph $\tree$ of the bipartite graph $(V\cup V',E)$, 
we let $\roott[\tree]$ denote its root node (note that its choice is 
arbitrary), $\leftt[\tree] = 
\{i\in V|~i\in \tree\}$ denote its source nodes and 
$\rightt[\tree] := \{j\in V'|~j\in \tree\}$ denote its target nodes. The 
number of nodes of $\tree$ is $\# \tree =
\# \leftt[\tree] + \# \rightt[\tree]$. 
Given $v \in \tree$, $\children[v]$ denotes the set of children nodes for $v$ 
in $\tree$, while $\parent[v]$ denotes the parent node of 
$v \neq \roott[\tree]$. Finally, given two nodes $u$ and $v$ in $\tree$, 
we write $u \preceq v$ (resp. $u \succeq v$) if $u$ belongs to the 
ancestor nodes (resp. successor nodes) of $v$.
\end{notation}

Equipped with this notation, consider a feasible transport plan $\tpplan$ for 
problem~\eqref{eq:uot} associated with a tree $\tree$, as well as two nodes 
$u \in \tree$ and $v \in \children[u]$. We let $\tpplan_{u \to v}$ denote the 
flow between both nodes, which is either $\tpplan_{u,v}$ if 
$u \in \leftt[\tree]$ or $\tpplan_{v,u}$ if $u \in \rightt[\tree]$. Assuming 
that we have computed dual potentials $(\ff,\gg)$ on the tree, we then compute 
a flow, or a transport, $\tpplan$ satisfying~\eqref{eq:dualfeasuot} on all 
nodes of $\tree$ save for the root node as follows.

\begin{proposition}
    \label{prop:comp-flow}
    Given a tree $\tree$ on the bipartite graph $(V\cup V',\mathcal E)$ and dual 
    potentials $(\ff,\gg)$ on $\tree$, 
    the flow $\tpplan$ defined by
    {\footnotesize
    \begin{equation}
        \label{eq:flow-from-potentials}
        \begin{multlined}
            \forall u\in\tree\backslash\roott[\tree],~
            \tpplan_{\parent[u]\to u}
            :=
            \begin{cases}
                \displaystyle 
                \sum_{j\in\rightt[\tree]:j\succeq u} (\nabla (\psi_2^*)(-\gg))_j - 
                \sum_{i\in\leftt[\tree]:i\succeq u} (\nabla (\psi_1^*)(-\ff))_i
                ~\text{if}~\parent[u]\in \leftt[\tree],\\
                \displaystyle 
                \sum_{i\in\leftt[\tree]:i\succeq u} (\nabla( \psi_1^*)(-\ff))_i -
                \sum_{j\in\rightt[\tree]:j\succeq u} (\nabla( \psi_2^*)(-\gg))_j
                ~\text{if}~\parent[u]\in \rightt[\tree],
            \end{cases}
        \end{multlined}
    \end{equation}}
    satisfies \eqref{eq:dualfeasuot} on all nodes $u\in\tree\backslash\roott[\tree]$.
\end{proposition}

\begin{proof}
	Let $\tpplan$ be a flow satisfying~\eqref{eq:dualfeasuot} given $(\ff,\gg)$. 
	At any node $u \in \tree\backslash\roott[\tree]$, the following recursive 
	relation holds:
	\begin{equation}
    \label{eq:rec-flow}
		\tpplan_{\parent[u]\to u}
		=
		\begin{cases}
    	\displaystyle
    	\nabla \psi_1^*(-\ff)_u - \sum_{v\in\children[u]}\tpplan_{u\to v}
    	~\text{if}~u\in \leftt[\tree],\\
    	\displaystyle
    	\nabla \psi_2^*(-\gg)_u - \sum_{v\in\children[u]}\tpplan_{u\to v}
    	~\text{if}~u\in \rightt[\tree].
		\end{cases}
	\end{equation}
	Applying the formula recursively from the leaves to the root gives the desired result.
\end{proof}

Algorithm~\ref{alg:compute-flow} implements the flow calculation 
described by Proposition~\ref{prop:comp-flow} based on the 
recursion~\eqref{eq:rec-flow}, and thus computes the 
flow~\eqref{eq:flow-from-potentials} in at most $\mathcal{O}(\#\tree)$ iterations. 
Algorithm~\ref{alg:compute-flow} can be viewed as a specialization of a method by
Ahuja et al~\cite[Alg. ``Comp. Flows'' p.300]{ahuja1988network} to 
the UOT setting.
It is not sufficient for our 
purpose, however, as its output is not guaranteed to be primal feasible 
(i.e. have non-negative coefficients) for problem~\eqref{eq:uot}. In the next 
section, we show how to correct this by updating the dual potentials, giving 
rise to a full network simplex method.

\begin{algorithm}[ht]
\SetAlgoLined
\KwIn{Tree $\tree$, dual potentials $(\ff,\gg)$ 
satisfying \eqref{eq:compuot} on $\tree$, initial flow $\tpplan$}
{\DontPrintSemicolon
\SetKwFunction{FMain}{FlowComputationNode}
\SetKwProg{Fn}{Function}{:}{}
\Fn{\FMain{$v$}}{
    $s\gets 0$\;
    \For{$u\in \children[v]$}{
        $s\gets s+\FMain{u}$\;
        }
        \uIf{$v\neq \roott[\tree]$}{
                $\tpplan_{\parent[v]\to v} \gets \nabla\psi_1^*(-\ff)_v-s$ if $v\in \leftt[\tree]$ else $\nabla\psi_1^*(-\gg)_v-s$\;\\
                \Return $\tpplan_{\parent[v]\to v}$
                }
                }}
                \FMain{$\roott[\tree]$}\;
                \par\noindent
                \KwOut{Flow $\tpplan$ satisfying \eqref{eq:dualfeasuot} associated to 
                $(\ff,\gg)$ on $\tree\backslash\roott[\tree]$}
                \caption{Flow computation on a tree}
                \label{alg:compute-flow}
            \end{algorithm}

\subsection{Network simplex for UOT}
\label{subsec:nsuot}

Suppose that we have computed a transport plan $\tpplan \ge 0$ and dual 
potentials $(\ff,\gg)$ satisfying~\eqref{eq:dualfeasuot} and \eqref{eq:compuot}. Either the triplet 
$(\tpplan,\ff,\gg)$ is a primal-dual solution of~\eqref{eq:rot-pd} or there 
exists an edge $(i_0,j_0)$ where dual feasibility is violated, i.e.
\begin{equation}
\label{eq:violatingedge}
     \ff_{i_0} + \gg_{j_0} 
    - \costmat_{i_0,j_0} =:\Delta > 0.
\end{equation}
We first explain how to update $\ff$ and $\gg$ when $i_0$ and $j_0$ belong 
to two distinct trees denoted by $\tree^-$ and $\tree^+$, respectively. 
Proposition~\ref{prop:comp-flow} provides a formula for adding the edge 
$(i_0,j_0)$ to the forest of two trees and updating the flow accordingly, 
but it does not guarantee that the resulting flow is nonnegative. Shifting 
the dual variables by a proper constant amount ensures that 
Proposition~\ref{prop:comp-flow} yields a nonnegative flow, and thus a 
primal-dual triplet that is optimal over the tree 
$\tree^+\cup (i_0,j_0) \cup \tree^-$.

We now describe how to compute a proper shift for each tree. Suppose that 
we replace the dual variables $(\ff,\gg)$ by 
\begin{equation}
\label{eq:shifteddual}
	(\widetilde \ff, \widetilde \gg) 
	:=
	\left\{
	\begin{array}{ll}
		(\ff-\gamma_-,\gg+\gamma_-) &\mbox{on $\tree^-$}, \\
		(\ff+\gamma_+,\gg-\gamma_+) &\mbox{on $\tree^+$}.
	\end{array}
	\right.
\end{equation}
Since the two trees are disjoint, the pair 
$(\widetilde \ff,\widetilde \gg)$ still satisfies~\eqref{eq:compuot}. 
However, applying Proposition~\ref{prop:comp-flow} to update $\tpplan$ based 
on $(\widetilde \ff,\widetilde \gg)$ rather than $(\ff,\gg)$ yields a 
different transport plan, which we denote by $\widetilde\tpplan$ in the rest 
of this section. By construction, the triplet 
$(\widetilde\tpplan, \widetilde\ff,\widetilde\gg)$ satisfies both 
complementarity~\eqref{eq:compuot} on $\tree^+\cup\tree^-$ and first-order 
conditions \eqref{eq:dualfeasuot} on  
$\tree^+\cup(i_0,j_0)\cup\tree^-$ regardless of the value of 
$\delta=\widetilde\tpplan_{i_0,j_0}-\tpplan_{i_0,j_0}$. We thus select 
$\delta$ so that both complementarity on $(i_0,j_0)$ and 
nonnegativity of $\widetilde \tpplan$ hold.
  
Applying Proposition~\ref{prop:comp-flow} on both $\tree^{-}$ and 
$\tree^{+}$ yields the following relationships between $\gamma_{-}$, 
$\gamma_+$ and $\delta$:
\begin{equation}
\label{eq:gamma-dela-general}
	\begin{cases}
	\displaystyle
		\tpplan_{i_0,j_0} + \delta = 
		\sum_{i\in \leftt[\tree^-]} (\psi_{1,i}^*)'(-\ff_i+\gamma_-)
		- \sum_{j\in \rightt[\tree^-]} (\psi_{2,j}^*)'(-\gg_j-\gamma_-),\\
	\displaystyle
		\tpplan_{i_0,j_0} + \delta = 
		\sum_{j\in \rightt[\tree^+]} (\psi_{2,j}^*)'(-\gg_j+\gamma_+)
		-  \sum_{i\in \leftt[\tree^+]} (\psi_{1,i}^*)'(-\ff_i-\gamma_+),
	\end{cases}
\end{equation}
where $\psi_1(\mathbf{\mu})=\sum_{i=1}^{n_1}\psi_{1,i}(\mu_i)$ (resp. 
$\psi_{2}(\mathbf{\nu})=\sum_{j=1}^{n_2} \psi_{2,j}(\nu_j)$) by 
separability. 
Thus, for any fixed $\delta \ge 0$,
 the values $\gamma_+$ and $\gamma_-$ 
satisfying~\eqref{eq:gamma-dela-general} are uniquely defined by 
strict convexity of $\psi_{1,i}$ and $\psi_{2,j}$. It follows that 
complementarity on $(i_0,j_0)$ is guaranteed by setting 
$\delta = \delta_{\max}$, where $\dmax$ solves the equation
\begin{equation}
\label{eq:delta-max}
	\gamma_+(\dmax) + \gamma_{-}(\dmax) = \Delta.
\end{equation}

If the transport plan obtained after pushing $\delta_{\max}$ along $(i_0,j_0)$ is not 
nonnegative, we must compute the smallest flow
$\delta_\mathrm{eff}:=\min\{\delta_\mathrm{eff}^+,\delta_\mathrm{eff}^-\}$
which cancels an edge.
Using  Proposition~\ref{prop:comp-flow}, one has: 
{\small
\begin{equation}
\label{eq:delta-eff-tm}
	\delta_\mathrm{eff}^-:=
	\min \left\{\delta\geq 0\bigg|
	\exists \parent[u]\in \leftt[\tree^-],
	\!\!\sum_{
        \begin{subarray}{c}
            j\in\rightt[\tree^-]\\
             j\succeq u
        \end{subarray}
        }\!\! (\psi_{2,j}^*)'(-\gg_j-\gamma_-(\delta)) 
	=  
    \!\!\sum_{\begin{subarray}{c}
            i\in\leftt[\tree^-]\\
        i\succeq u
        \end{subarray}}\!\!
    (\psi_{1,i}^*)'(-\ff_i+\gamma_-(\delta)) 
	\right\}
\end{equation}
}%
for the flow on $\tree^-$ and    
{\small
\begin{equation}
\label{eq:delta-eff-tp}
	\delta_\mathrm{eff}^+ :=
	\min \left\{\delta\geq 0\bigg|
	\exists \parent[u]\in \rightt[\tree^+],
    \!\!\sum_{\begin{subarray}{c}
        i\in\leftt[\tree^+]\\
        i\succeq u
    \end{subarray}
        }\!\!
     (\psi_{1,i}^*)'(-\ff_i-\gamma_+(\delta)) 
	=
	\!\!
    \sum_{\begin{subarray}{c}j\in\rightt[\tree^+]\\j\succeq u\end{subarray}}
    \!\!(\psi_{2,j}^*)'(-\gg_j+\gamma_+(\delta)) 
	\right\}
\end{equation}
}%
for that on $\tree_+$. We let $(u_{\mathrm{eff}}^\pm, v_{\mathrm{eff}}^\pm)$ 
denote the associated arcs (taking the first in lexicographic ordering to avoid 
cycling). Negative transport plan values only occur when 
$\delta_{\mathrm{eff}} < \delta_{\max}$, in 
which case one must 
remove the corresponding arc. If parts of the subtree become isolated,
they no longer need updating.  This process can be repeated on the updated 
$\tree^\pm$, until $\Delta=0$.
%
Algorithm~\ref{alg:optimizing-union} summarizes the procedure for disjoint 
trees. As shown below, the procedure eventually terminates because the number 
of edges decreases at every iteration.

\begin{algorithm}[ht]
\SetAlgoLined
\KwIn{Triplet $(\tpplan,\ff,\gg)$, edge $(i_0,j_0)$ and associated 
trees $\tree^\pm$ so that $\tpplan_{ij} >0$, \eqref{eq:dualfeasuot} and 
\eqref{eq:compuot} hold on $\tree^\pm$. Let $k\gets 0$, $(\tpplan^{(0)},\ff^{(0)},\gg^{(0)})\gets (\tpplan,\ff,\gg)$
}
\While{$\Delta^{(k)}:=\ff_{i_0}^{(k)}+\gg_{j_0}^{(k)}-\costmat_{i_0,j_0}>0$}{
    Compute $\dmax$, $\delta_\mathrm{eff}^\pm$ and 
    $(u_\mathrm{eff}^\pm,v_\mathrm{eff}^\pm)$ using \eqref{eq:delta-max}, 
    \eqref{eq:delta-eff-tm} and \eqref{eq:delta-eff-tp}\;\\
    $\delta \gets \min(\dmax,\delta_\mathrm{eff}^+,\delta_\mathrm{eff}^-)$\;\\
    $(\ff^{(k+1)},\gg^{(k+1)})\gets (\ff^{(k)}\pm\gamma_\pm(\delta),\gg^{(k)}\mp \gamma_\pm(\delta))$ on $\tree^\pm$\;\\
    Compute $\tpplan^{(k+1)}$ from $(\ff^{(k+1)},\gg^{(k+1)})$ using \eqref{eq:rec-flow}, and
    $\tpplan^{(k+1)}_{i_0,j_0}\gets \tpplan^{(k)}_{i_0,j_0}+\delta$\;\\
    \If{$\delta<\delta_{\mathrm{max}}$}{
        \uIf{$\delta_\mathrm{eff}^+\le\delta_\mathrm{eff}^-$}{
        Remove the edge $(u_\mathrm{eff}^+,v_\mathrm{eff}^+)$ 
        and update $\tree^+$ 
        accordingly\;
        }
        \uIf{$\delta_\mathrm{eff}^+\ge\delta_\mathrm{eff}^-$}{
        Remove the edge $(u_\mathrm{eff}^-,v_\mathrm{eff}^-)$ 
        and update $\tree^-$ 
        accordingly\;
        }
    }
    $k\gets k+1$\;
}
\KwOut{Optimal variables $(\tpplan^{(k)},\ff^{(k)},\gg^{(k)})$}
\caption{UOT on a union of two trees}
\label{alg:optimizing-union}
\end{algorithm}

\begin{proposition}
\label{prop:nbopsuniontrees}
    Algorithm \ref{alg:optimizing-union} converges to an optimal solution 
    $(\tpplan,\ff,\gg)$ in $\mathcal{O}(\#\tree^- + \#\tree^+)$
    iterations.
\end{proposition}

\begin{proof}
	Suppose that we reach the $k$th iteration of 
	Algorithm~\ref{alg:optimizing-union}. Two cases must be considered.
	Suppose first that $\delta = \dmax \le 
    \delta_\mathrm{eff}
     $. In that case, there is no edge removal, 
	and the update formulas guarantee $\Delta^{(k+1)}=0$, so that the 
	method stops. Moreover, since the updates preserve first-order 
	and complementarity conditions, the triplet 
	$(\tpplan^{(k+1)},\ff^{(k+1)},\gg^{(k+1)})$ is optimal on 
	$\tree^+\cup (i_0,j_0)\cup\tree^-$, hence the result holds.
	
	Suppose now that $\delta=
    \delta_\mathrm{eff}
    < \dmax$. Then this step will result in 
	removing one edge from $\tree^+$ in case
	$\delta= \delta_\mathrm{eff}^+$ and/or one
	edge from $\tree^-$ if
	$\delta= \delta_\mathrm{eff}^-$.
	We obtain a new triplet $(\tpplan^{(k+1)},\ff^{(k+1)},\gg^{(k+1)})$ 
	with 
	$\Delta^{(k+1)} < \Delta^{(k)}$. Let $(u,v)$ denote an edge 
	removed during the iteration. If this edge is removed 
	from $\tree^-$, then $u \in \leftt[\tree^-]$ as shown in 
	Section~\ref{subsec:nsuot} while $v$ is removed from
	 $\rightt[\tree^-]$. Moreover, $\ff_u^{(k+1)}+\gg_v^{(k+1)}=\costmat_{u,v}$ 
	since complementarity conditions are preserved, while any subsequent 
	iteration will only reduce $\ff_u$ while keeping $\gg_v$ at the same 
	value. Thus, dual feasibility is guaranteed on $(u,v)$ for all 
	subsequent iterations, while
	one easily checks that conditions~\eqref{eq:dualfeasuot} are satisfied in the old part of $\tree^-$, attached
	to $v$, which has been pruned after the update and
	will not be touched again in the remaining steps. A similar reasoning applies when $(u,v)$ is 
	an edge removed from $\tree^+$.

	Overall, for every iteration of Algorithm~\ref{alg:optimizing-union}, 
	either the method terminates or it removes one edge in the 
	worst case (possibly more edges and vertices) while preserving 
	dual feasibility. Since the number of edges that can be removed 
	is at most $(\# \tree^- + \# \tree^+)$, it follows that the 
	method cannot perform more than $(\# \tree^- + \# \tree^+)$ 
	iterations. Moreover, if the maximum of iterations is reached, 
	then the last remaining edge must be $(i_0,j_0)$. The updating 
	formula then guarantees that we fall into the first case above, 
	and optimality follows.

\end{proof}

\begin{remark}
In the rest of the paper, we will assume that each of the quantities 
$\delta_\mathrm{eff}^\pm$ can be computed in $\mathcal{O}(\# \tree^\pm)$ 
iterations and $\delta_\mathrm{max}$ in $\mathcal{O}(1)$ iterations. While 
this assumption may appear restrictive, it holds for the special cases that 
we describe in Section~\ref{sec:particular-cases} below (UOT and SUOT with 
$L^2$ and $\KL$ relaxation). In general, one must resort to numerical solvers 
to compute these values, which we illustrate for $L^4$ relaxation functions 
in our numerical experiments.
\end{remark}

We can now describe our full adapted network simplex framework for unbalanced 
optimal transport, given in Algorithm~\ref{alg:adapted-ns}. At every 
iteration, provided the current triplet is not optimal, the method finds a 
pivot edge $(i_0,j_0)$ and calls Algorithm~\ref{alg:optimizing-union} to produce a new 
primal-dual iterate. Note that when $i_0$ and $j_0$ belong to the same tree, 
one must first proceed as in the standard network simplex: adding $(i_0,j_0)$ creates a cycle in the
tree along which one pushes flow according 
to~\eqref{eq:push-cycle}. This cancels another edge and creates a new tree $\tree$.
One then defines $\tree^-,\tree^+$ as the two connected
components of $\tree\setminus \{(i_0,j_0)\}$
respectively containing $i_0$ and $j_0$, which
are set as roots of the corresponding subtree.
We then run Algorithm~\ref{alg:optimizing-union}, the only difference
being that the initial flow $\tpplan_{i_0,j_0}$ is
not zero (which does not affect the method's correctness).

\begin{algorithm}[ht]
\SetAlgoLined
\KwIn{Triplet $(\tpplan,\ff,\gg)$ with $\tpplan\geq 0$ and 
$(\ff,\gg)$ complementary potentials}
\While{$(\tpplan,\ff,\gg)$ not optimal for~\eqref{eq:rot-pd}}{
	Select a pivot edge $(i_0,j_0)$ satisfying~\eqref{eq:violatingedge}\;\\
	\If{$i_0$ and $j_0$ belong to the same tree}{
		Add the edge $(i_0,j_0)$ to the tree, apply the update~\eqref{eq:push-cycle}, 
		and remove one edge from the original tree.
		Define $\tree^\pm$ as the connected components
		of the new tree minus $(i_0,j_0)$\;
	}
    Update $(\tpplan,\ff,\gg)$ by calling Algorithm~\ref{alg:optimizing-union} 
    on $\tree^- \cup (i_0,j_0)\cup \tree^+$, where $i_0$
    is set as the root of $\tree^-$ and 
    $j_0$ as the root of $\tree^+$.\;
}
\KwOut{$(\tpplan,\ff,\gg)$ optimal for~\eqref{eq:rot-pd}}
\caption{Adapted network simplex for UOT}
\label{alg:adapted-ns}
\end{algorithm}

Every iteration of Algorithm~\ref{alg:adapted-ns} produces a triplet that 
is optimal on the current forest corresponding to the support of $\tpplan$. 
Since there exists a finite number of forests and the primal cost decreases 
at every iteration, the method must converge to a solution in a finite number 
of iterations. 
Note that Algorithm~\ref{alg:adapted-ns} relies on a pivot rule to select an 
edge violating dual-feasibility, \emph{i.e.} satisfying \eqref{eq:violatingedge}, which will be discussed in the next 
sections. Beforehand, we explain how to compute our quantities of interest 
in closed form for special UOT instances.

\section{Special cases of unbalanced formulations}
\label{sec:particular-cases}

In this section, we describe several instances of the unbalanced optimal 
transport problem and explain how they allow for direct calculation of the 
formulas used in Algorithm~\ref{alg:adapted-ns}. We study formulations based 
on $L^2$ and $\KL$ relaxation functions, then describe semi-unbalanced 
optimal transport problems as a limiting case of unbalanced optimal transport.

\subsection{UOT with quadratic relaxation}
\label{subsec:l2uot}

Arguably the most classical relaxation functions in optimal transport are 
weighted quadratic (or $L^2$) relaxation functions. Given two 
discrete distributions $\aalpha^{\oone}$, $\aalpha^{\ttwo}$ and two positive 
weight vectors $\ww^{\oone}$, $\ww^{\ttwo}$, the relaxation functions are given 
by 
\begin{equation}
\label{eq:psiL2}
    \forall k \in \IntInterval{1}{2},\ \forall \bm{x} \in\R^{n_k}, 
    \quad 
    \psi_k(\bm{x}) = 
    \frac{1}{2\tau^k} 
    \normad{\sqrt{\ww^{k}} 
    \boldsymbol{\odot} \left(\bm{x}-\aalpha^{k}\right)}_2^2,
\end{equation}
where $\tau^{k}>0$ are two positive values. The corresponding instance of 
problem~\eqref{eq:rot}, termed UOT-$L^2$, is
\begin{equation}
\label{eq:rot-l2}
	\min_{(\mmu,\nnu) \in \R^{n_1+n_2}}\ 
    \min_{\tpplan\in\margset}
    \tpplan\!:\!\costmat 
    + \frac{1}{2\tau^{1}}
    \normad{\sqrt{\ww^{\oone}} \boldsymbol{\odot}
    \left(\mmu-\aalpha^{\oone}\right)}_2^2
    + \frac{1}{2\tau^{2}}
    \normad{\sqrt{\ww^{\ttwo}} \boldsymbol{\odot}
    \left(\nnu-\aalpha^{\ttwo}\right)}_2^2.
\end{equation}
Note that the measures $\aalpha^1$ and $\aalpha^2$ are not required to have 
the same mass, and thus the formulation~\eqref{eq:rot-l2} applies to pairs of 
measures between which no optimal transport can be constructed. 
Problem~\eqref{eq:rot-l2} is equivalent to a LASSO 
problem~\cite{chapel2021unbalanced}, which can be tackled efficiently using 
dedicated solvers such as \celer\cite{massias2018celer}.

Algorithm~\ref{alg:adapted-ns} applied to problem~\eqref{eq:rot-l2} can be 
initialized with 
\begin{equation*}
\label{eq:init:l2uot}
	\tpplan = \bm{0}_{n_1\times n_2},
	\quad
	\ff = \tfrac{1}{\tau^1}\ww^{\oone} \boldsymbol{\odot}\aalpha^{\oone},
	\quad
    \gg = \tfrac{1}{\tau^2}\ww^{\ttwo}\boldsymbol{\odot}\aalpha^{\ttwo}.
\end{equation*} 
Indeed, these 
satisfy $\tpplan \ge \bm{0}$, 
\eqref{eq:dualfeasuot} and~\eqref{eq:compuot}. In addition, the 
functions~\eqref{eq:psiL2} are strictly convex and separable with
$(\psi_{1,i}^*)'(y) = \aalpha_i^{\oone} + \frac{\tau^1}{\ww_i^{1}}y$ 
and $(\psi_{2,j}^*)'(z) = \aalpha_j^{\ttwo} + \frac{\tau^2}{\ww_j^{2}}z$, 
%
the general relationship~\eqref{eq:gamma-dela-general} yields 
closed-form expressions for $\gamma_\pm(\delta)$, $\dmax$ and 
$\delta_{\mathrm{eff}}^{\pm}$:
\begin{subequations}
\label{eq:qoi-uot-l2}
	\begin{align}
		\displaystyle
        \label{eq:gamma-delta-uot-l2}
		\gamma_\pm(\delta) &= 
		\frac{\delta }{\lambda(\roott[\tree^\pm])},\\
		\displaystyle
		 \dmax &=
    	\Delta\left(\frac{1}{\lambda(\roott[\tree^+])}
    	+\frac{1}{\lambda(\roott[\tree^-])}\right)^{-1},\\
    	\displaystyle
    	\delta_\mathrm{eff}^+ &=
        \lambda(\roott[\tree^+])
        \min_{\parent[u]\in\rightt[\tree^+]}
        \frac{\tpplan_{\parent[u]\to u}}{\lambda(u)},\\
		\displaystyle
		\delta_\mathrm{eff}^- &=
		\lambda(\roott[\tree^-])
		\min_{\parent[u]\in\leftt[\tree^-]}
		\frac{\tpplan_{\parent[u]\to u}}{\lambda(u)}.
	\end{align}
\end{subequations}
where:
    \begin{equation}\label{eq:deflambda}
        \lambda(u):= \tau^{1}\sum_{\leftt[\tree]\ni i\succeq u}
            \frac{1}{\ww^{\oone}_i}
            +
            \tau^{2}\sum_{\rightt[\tree]\ni j\succeq u} 
            \frac{1}{\ww^{\ttwo}_j}.
    \end{equation}

\begin{remark} 
\label{rem:l1-norm-l2}
In the UOT-$L^2$ setting, the bound of Proposition \ref{lem:l1-norm-solutions} 
is $\lVert \aalpha^{\oone}\rVert_1+\lVert \aalpha^{\ttwo}\rVert_1$. This bound is asymptotically 
optimal. Indeed, consider the values $\aalpha^{\oone}=\aalpha^{\ttwo} = (n^2,1,\dots,1)^T \in \R^n$, 
$\tau=1$ and $\costmat_{ij} = 2n\,\delta_{i=1,j=1}$. One may then check that 
$\tpplan := \left(\tfrac{n^2+1}{n+1}~\text{if}~(i=1,j\ge 2)~\text{or}~(i\ge 2,j= 1) \right)_{ij}$ 
is optimal and $\tfrac{\lVert \tpplan\rVert_1}{\lVert \aalpha^{\oone}\rVert_1+\lVert \aalpha^{\ttwo}\rVert_1}\to 1$
as $n\to\infty$.
\end{remark}

\subsection{UOT with Kullback-Leibler relaxation}
\label{subsec:kluot}

Another popular choice of relaxation, which we refer to as the 
Kullback-Leibler relaxation, relies on the Bregman divergence associated with 
the entropy function:
\begin{equation*}
\label{eq:kl}
    \forall (\bm{x},\bm{y})\in(\R^n_+)^2,~\KL(\bm{x},\bm{y}):=
    \sum_{i=1}^n x_i \log\left(\frac{x_i}{y_i}\right) - x_i + y_i,
\end{equation*}
considering that $0 \log(0)=0$ (and the value is
infinite if $y_i=0$ and $x_i>0$ for some $i$). Given two measures $\aalpha^{\oone}$ and 
$\aalpha^{\ttwo}$ as well as two positive parameters $\tau^1,\tau^2$, we set 
\begin{equation}
    \label{eq:def-kl-psi}
    \forall k\in\IntInterval{1}{2},~\forall \bm{x}\in\R_+^n,~\psi_k(\bm{x}) =
    \frac{1}{\tau^k}\KL(\bm{x},\aalpha^{\bm k})
\end{equation}
and we obtain the following instance of problem~\eqref{eq:rot}, termed 
UOT-$\KL$ in the rest of the paper:
\begin{equation}
  \label{eq:rot-kl}
    \min_{(\mmu,\nnu) \in \R^{n_1+n_2}}
    \min_{\tpplan\in\margset}
    \tpplan\!:\!\costmat 
    + \frac{1}{\tau^{1}}\KL(\mmu,\aalpha^{\oone})
    + \frac{1}{\tau^{2}}\KL(\nnu,\aalpha^{\ttwo}).
\end{equation}
Here again, we do not enforce that $\aalpha^{\oone}$ and $\aalpha^{\ttwo}$ 
have identical mass. The formulation~\eqref{eq:rot-kl} has been employed 
for multi-label classification~\cite{frogner2015learning}. 

Explicit formulas for Algorithm~\ref{alg:adapted-ns} can be obtained when 
$\tau^1=\tau^2=\tau>0$. Since 
\begin{equation*}
    \forall k \in \IntInterval{1}{2},~
    \forall i\in\IntInterval{1}{n_k},
    ~\forall \bm{y} \in\R^{n_k},~(\psi_{k,i}^*)'(\bm y) = 
    \aalpha_i^{\bm k}\exp(\tau \bm y_i),
\end{equation*}
a valid initialization of the algorithm is given by
\begin{equation*}
\label{eq:init:kluot}
	\begin{cases}
		\tpplan = \diag{\aalpha^{\oone}\exp(-\tau \ff)},\\
		\ff= \left(\tau^{-1}\log\left(
		\aalpha^{\oone}\boldsymbol{\oslash}\aalpha^{\ttwo}\right)
		+\diag{\costmat}\right)/2,\\
		\gg = \diag{\costmat} - \ff.
	\end{cases}
\end{equation*}
The equations \eqref{eq:gamma-dela-general} reduce to
\begin{equation*}
    \begin{cases}
    \displaystyle
        \tpplan_{i_0,j_0} + \delta = 
        \sum_{i\in \leftt[\tree^-]} \aalpha_i^{\bm 1}\exp(\tau(-\ff_i+\gamma_-))
        - \sum_{j\in \rightt[\tree^-]} \aalpha_j^{\bm 2}\exp(\tau(-\gg_j-\gamma_-)),\\
         \displaystyle
        \tpplan_{i_0,j_0} + \delta = 
        \sum_{j\in \rightt[\tree^+]} \aalpha_j^{\bm 2}\exp(\tau(-\gg_j+\gamma_+))
        -  \sum_{i\in \leftt[\tree^+]} \aalpha_i^{\bm 1}\exp(\tau(-\ff_i-\gamma_+)),
    \end{cases}
\end{equation*}
which, combined with \eqref{eq:dualfeasuot}, can be simplified into:
\begin{equation}
\label{eq:gamma-pm-kl-impl}
	\delta = 2\sinh(\tau\gamma_\pm) \sigma_\pm
	+(\exp(\tau\gamma_\pm)-1) \tpplan_{i_0,j_0}.
\end{equation}
where $\sigma_\pm:= \sigma(\roott[\tree^\pm])$ and $\sigma(u) := \sum_{\parent[v]\succeq u}\tpplan_{\parent[v]\to v}$. The following result shows that 
$\gamma_{\pm}(\delta)$ can be computed in closed form.

\begin{proposition}
\label{prop:gamma-delta-full-kl}
	Consider problem~\eqref{eq:uot} 
	with the functions in~\eqref{eq:def-kl-psi}, for $\tau^1=\tau^2=:\tau$. Then, the various quantities computed at each 
	iteration of Algorithm~\ref{alg:optimizing-union} are explicitly defined as 
	follows:
    \begin{subequations}
    	\begin{align}
    		\gamma_\pm(\delta) 
            &=\frac{1}{\tau}
            \left(
            \mathrm{arcsinh}\left(
                \frac{\delta+\tpplan_{i_0,j_0}}{2\sqrt{\sigma_\pm\widetilde{\sigma}_\pm}}
                \right)-\frac{1}{2}\ln\left(\frac{\widetilde{\sigma}_\pm}{\sigma_\pm}\right)
            \right)
        	\label{eq:gammapm-fullkl} \\
        	\dmax 
        	&= 
        	\frac{
        	e^{\tau\Delta}\widetilde \sigma_+ \widetilde \sigma_- 
        	- e^{-\tau\Delta}\sigma_+\sigma_-
	        }{
            \sqrt{(\widetilde\sigma_++e^{-\tau\Delta}\sigma_-)
            (e^{\tau\Delta}\widetilde\sigma_-+\sigma_+)}            
            }
        	- \tpplan_{i_0,j_0}
        	\label{eq:deltamax-fullkl} \\
        	\delta_{\mathrm{eff}}^+ 
        	&=
        	-  \tpplan_{i_0,j_0}+ 
            \min_{\parent[u]\in\rightt[\tree^+]}
            \widetilde\sigma_+ U_{\parent[u]\to u} 
            - \frac{\sigma_+}{U_{\parent[u]\to u}} 
            \label{eq:deltaeffp-fullkl}\\
            \delta_{\mathrm{eff}}^- 
            &=
            -\tpplan_{i_0,j_0}+ 
            \min_{\parent[u]\in\leftt[\tree^-]}
            \widetilde\sigma_- U_{\parent[u]\to u} 
            - \frac{\sigma_-}{U_{\parent[u]\to u}},
        	\label{eq:deltaeffm-fullkl} 
        \end{align}
    \end{subequations}
    where $\widetilde \sigma_\pm := \sigma_\pm + \tpplan_{i_0,j_0}$ 
    and $U_{\parent[u]\to u} := 
    \sqrt{1+\frac{\tpplan_{\parent[u]\to u}}{\sigma(u)}}$.
\end{proposition}

\begin{proof}
    One readily checks by direct substitution that \eqref{eq:gammapm-fullkl} satisfies \eqref{eq:gamma-pm-kl-impl}.
    Since $2\mathrm{sinh}(\tau\gamma_\pm) =e^{\tau\gamma_\pm}- e^{-\tau\gamma_\pm}$, 
    \eqref{eq:gamma-pm-kl-impl} can be written:
    \begin{equation}
        \label{eq:dmax-kl-rewritten}
        \dmax+\tpplan_{i_0,j_0} = \widetilde\sigma_\pm e^{\tau\gamma_\pm}-\sigma_\pm e^{-\tau\gamma_\pm}
    \end{equation} 
    and using \eqref{eq:delta-max} gives 
    $e^{2\tau\gamma_+} =
    (\widetilde\sigma_- e^{\tau\Delta}+\sigma_+)/(\widetilde\sigma_+ + e^{-\tau\Delta} \sigma_-)
    $. One obtains $\dmax$ by substituting the latter in \eqref{eq:dmax-kl-rewritten}.
    Finally, we show~\eqref{eq:deltaeffm-fullkl}. Consider $u$ such that $\parent[u]\in\leftt[\tree^-]$, 
    the flux $\tpplan_{\parent[u]\to u}^\gamma$ vanishes if and only if:
    \begin{equation}
        2\sinh(-\tau\gamma_-)\sigma(u) 
		+ \exp(-\tau\gamma_-)\tpplan_{\parent[u]\to u} 
		= 0
    \end{equation}
     from which we obtain $\gamma_- = \ln\left(U_{\parent[u]\to u}\right)/\tau$.
    Substituting $\gamma_-$ in \eqref{eq:gamma-pm-kl-impl} gives 
    $\delta=\widetilde\sigma_- U_{\parent[u]\to u} 
    - \tfrac{\sigma_-}{U_{\parent[u]\to u}}-\tpplan_{i_0,j_0}$, 
    while~\eqref{eq:deltaeffm-fullkl} is obtained by taking the minimum on all 
    decreased flows. The same reasoning shows \eqref{eq:deltaeffp-fullkl}.
\end{proof}

\begin{remark}
    \label{rem:upd-flow-l2-kl}
    Updates from flow pushing can also be written explicitly under the 
    assumptions of Proposition~\ref{prop:gamma-delta-full-kl}.
    Let $(\ff,\gg,\tpplan)$ satisfy~\eqref{eq:dualfeasuot} on 
    $\tree\backslash\roott[\tree]$ and suppose that we want to compute 
    the flow $ \tpplan^\gamma$ satisfying \eqref{eq:dualfeasuot}
	from $( \ff^\gamma,  \gg^\gamma) := (\ff+\gamma,\gg-\gamma)$
    given some constant $\gamma$. 
    In the $L^2$ case \eqref{eq:rot-l2}, the flow $\tpplan$ is given by:
    {\begin{equation*}
    \label{eq:flux-update-L2}
    	\forall u\in\tree\backslash\roott[\tree],\quad
    	\tpplan_{\parent[u]\to u}^\gamma =
    	\begin{cases}
    	\displaystyle
     	\tpplan_{\parent[u]\to u} 
      	+ \gamma\lambda(u)~\text{if}~\parent[u]\in \leftt[\tree],\\
		\displaystyle
		\tpplan_{\parent[u]\to u} 
		- \gamma \lambda(u)~\text{if}~\parent[u]\in \rightt[\tree].
        \end{cases} 
    \end{equation*}}
    where $\lambda$ is defined in~\eqref{eq:deflambda}.
    In the $\KL$ case \eqref{eq:rot-kl}, the flow is given by:
    {\begin{equation}
    \label{eq:flux-update-kl}
		\forall u\in\tree\backslash\roott[\tree],~
		\tpplan_{\parent[u]\to u}^\gamma =
		\begin{cases}
        	\displaystyle
        	\left(e^{\tau^{2}\gamma}-e^{-\tau^{1}\gamma}\right)\sigma(u)
        	+e^{\tau^{2}\gamma}\tpplan_{\parent[u]\to u}
        	~\text{if}~\parent[u]\in \leftt[\tree],\\
        	\displaystyle
        	\left(e^{-\tau^{1}\gamma}-e^{\tau^{2}\gamma}\right)\sigma(u)
        	+e^{-\tau^{1}\gamma}\tpplan_{\parent[u]\to u}
        	~\text{if}~\parent[u]\in \rightt[\tree].
    	\end{cases}
	\end{equation}}
\end{remark}

\begin{remark}
    \label{rem:norm-kl}
    In  \eqref{eq:rot-kl}, we do not constrain $\mmu,\nnu$ to be probability distributions so that $\tpplan$ is not constrained to be in the simplex. 
    However, a 
    minimizer on the probability simplex is easily found from a
    solution of \eqref{eq:rot-kl} through renormalization.
    Given a solution $(\tpplan^*,\ff^*,\gg^*)$ of \eqref{eq:rot-kl} and a value $\eta>0$, define the normalized transport plan $\tpplan' := \tpplan^*/\eta$. We can write $\tpplan' := \tpplan^*\exp(-\mu \indicm_{n_1}\indicm_{n_2}^\top)$ with $\mu :=\log(\eta)$. Since $(\tpplan^*,\ff^*,\gg^*)$ satisfies first-order conditions, we have:
     \begin{equation}
     \begin{cases}
         \tpplan'\indic &= \aalpha^{\oone}\exp(-\tau^1\ff')
         \\
         (\tpplan')^\top\indic &= \aalpha^{\ttwo}\exp(-\tau^2\gg')
     \end{cases}
     \end{equation}
     where $(\ff',\gg') := (\ff^*+(\mu/\tau^1)\indicm_{n_1},\gg^*+(\mu/\tau^2)\indicm_{n_2})$. By dual-feasibility and complementarity of $(\tpplan^*,\ff^*,\gg^*)$, we have $\ff'\oplus\gg'\leq \costmat+\bm{\zeta} $ 
     and $(\costmat+\bm{\zeta}-\ff'\oplus\gg')\odot \tpplan' = \mathbf{0}_{n_1\times n_2}$ 
     where $\bm{\zeta}:=\mu(1/\tau^1+1/\tau^2)\indicm_{n_1}\indicm_{n_2}^\top$. 
     Hence, choosing $\eta=\lVert\tpplan^*\rVert_1$, $(\tpplan',\ff',\gg')$ is a solution of:
     \begin{equation}
    \min_{(\mmu,\nnu) \in \R^{n_1+n_2}
    }
    \min_{
    \substack{\tpplan\in\margset\\ \lVert\tpplan\rVert_1 = 1}}
    \tpplan\!:\!\costmat 
    + \frac{1}{\tau^1}\KL(\mmu,\aalpha^{\oone})
    + \frac{1}{\tau^2}\KL(\nnu,\aalpha^{\ttwo})
    \end{equation}
and $\mu(1/\tau^1+1/\tau^2)$ is then the Lagrange multiplier for 
$\lVert\tpplan\rVert_1 = 1$. 
While a similar result exists in the literature for $\tau^1=\tau^2$~\cite[Lem. 1]{le2021robust}, 
our primal-dual formulation leads to a simpler proof.
\end{remark}

\begin{remark}
    \label{rem:l1-norm-kl}
    As in UOT-$L^2$, the vectors $\bm{m}^{\bm k}$ of Proposition~\ref{lem:l1-norm-solutions} are the distributions 
    $\aalpha^{\bm k}$. 
    However, now 
    this bound is not tight. Let $(\tpplan,\ff,\gg)$ satisfy the same
    assumptions as in Proposition \ref{lem:l1-norm-solutions}. We have $0\ge- \costmat\!:\!\tpplan = -(\ff\oplus \gg)\!:\!\tpplan$ by complementarity
    and $0\ge-(\ff\oplus \gg)\!:\!\tpplan = -\langle \ff,\tpplan\indicm_{n_2}\rangle-\langle \gg,\tpplan^\top\indicm_{n_1}\rangle=\tfrac{1}{\tau^1}
    \langle\log\left({\tpplan\indicm_{n_2}}/{\aalpha^{\oone}}\right),\tpplan\indicm_{n_2}\rangle+\tfrac{1}{\tau^2}
    \langle\log\left({\tpplan^\top\indicm_{n_1}}/{\aalpha^{\ttwo}}\right),\tpplan^\top\indicm_{n_1}\rangle$ by first order conditions.
    By log-sum inequality, we deduce $0\ge \tfrac{\lVert\tpplan\rVert_1}{\tau^1} \log\left({\lVert\tpplan\rVert_1}/{\lVert\aalpha^{\oone}\rVert_1}\right)+
    \tfrac{\lVert\tpplan\rVert_1}{\tau^2}\log\left({\lVert\tpplan\rVert_1}/{\lVert\aalpha^{\ttwo}\rVert_1}\right)$ and finally:
    \begin{equation}
        \label{eq:bound-kl}
        \lVert\tpplan\rVert_1 \le \lVert\aalpha^{\oone}\rVert_1^{\ell}\lVert\aalpha^{\ttwo}\rVert_1^{1-\ell}
        ~\text{with}~\ell:=\tfrac{\tau^2}{\tau^1+\tau^2}.
    \end{equation}
    This bound is optimal. To see this, let $a_k:=\lVert\aalpha^{\bm k}\rVert_1$ and define $\rho:=a_1^\ell a_2^{1-\ell}$. When $\costmat=\mathbf{0}_{n_1n_2}$, the plan $\tpplan := \tfrac{m}{a_1a_2}\aalpha^{\bm 1}(\aalpha^{\bm 2})^\top$ is optimal as it satisfies feasibility and (together with associated potentials) dual feasibility and complementarity.
    Furthermore, $\lVert\tpplan\rVert_1 = \rho$ and reaches the bound \eqref{eq:bound-kl}.
    Note that the bound $\tfrac{1}{2}\left(\lVert\aalpha^{\oone}\rVert_1+\lVert\aalpha^{\ttwo}\rVert_1\right)$ derived by Nguyen et 
    al~\cite[Lem.9]{nguyen2023unbalanced} in the case $\tau^1=\tau^2$ is  
    better that our general 
    bound of Proposition \ref{lem:l1-norm-solutions} but weaker than \eqref{eq:bound-kl}.
\end{remark}

\subsection{Semi-unbalanced optimal transport}
\label{subsec:suot-particular-cases}

Semi-unbalanced optimal transport (SUOT) consists in relaxing only one of the 
marginal constraints, and can be understood as the proximal operator of a 
transport distance~\cite{LionsMercier}. In this paper, we are interested in 
the semi-unbalanced version of~\eqref{eq:rot-l2} and~\eqref{eq:rot-kl}. These 
formulations are obtained by letting $\tau^1 \rightarrow 0$ in the 
definition of $\psi_1$, constraining the first marginal. 

\paragraph{SUOT-$L^2$}
Consider first the semi-unbalanced counterpart of~\eqref{eq:rot-l2} where we 
enforce the constraint $\tpplan\indic = \aalpha^{\oone}$ (i.e. $\tau^1=0$). 
The resulting problem has the form
\begin{equation}
\label{eq:srot-l2}
	\min_{\nnu \in \R^n}
	\min_{\tpplan\in\mathbf{U}(\aalpha^{\oone},\nnu)} \tpplan\!:\! \costmat
	+ \frac{1}{2\tau^2}
	\normad{\sqrt{\ww^{\ttwo}}\odot(\nnu-\aalpha^{\bm 2})}_2^2.
\end{equation}
Note that if $\|\aalpha^{\oone}\|_1 = 1$, then the solution 
of~\eqref{eq:srot-l2} is a proper transport plan, and one recovers the optimal 
transport formulation by setting $\tau^2=0$ if $\|\aalpha^{\ttwo}\|_1 = 1$. Interestingly, the solution 
of~\eqref{eq:srot-l2} can be viewed as a proximal operator for the optimal 
transport distance, which can then be employed in proximal splitting 
algorithms such as FISTA~\cite{beck2009fast,MBlondel_VSeguy_ARolet_2018}. 
We will illustrate this further in Section~\ref{sec:num-exp}.

A proper initialization for Algorithm~\ref{alg:adapted-ns} applied 
to~\eqref{eq:srot-l2} is given by
\begin{equation*}
	\begin{cases}
		\gg =\tfrac{1}{\tau^2}\ww^{\ttwo}
		\boldsymbol{\odot}(\aalpha^{\ttwo}-\aalpha^{\oone}),\\
		\ff = \diag{\costmat}-\gg,\\
        \tpplan = \diag{\aalpha^{\oone}}.
	\end{cases}
\end{equation*}  
while the quantities of interest for Algorithm~\ref{alg:optimizing-union} can 
be computed by setting $\tau^1=0$ in formulas~(\ref{eq:qoi-uot-l2}-\ref{eq:deflambda}).

\paragraph{SUOT-$\KL$} 
One can also set $\tau^1=0$ in~\eqref{eq:rot-kl}, in which case the 
first-order conditions~\eqref{eq:dualfeasuot} reduce to 
\begin{equation*}
    \begin{cases}
        \tpplan\indic=\aalpha^{\oone},\\
        \tpplan^\top\indic = \aalpha^{\ttwo}\exp(-\tau^2 \gg),
    \end{cases}
\end{equation*}
so that a valid initialization is
\begin{equation*}
    \begin{cases}
        \gg = -\frac{1}{\tau^2}\log(\aalpha^{\oone}
        \boldsymbol{\oslash}\aalpha^{\ttwo}),\\
        \ff = \diag{\costmat}-\gg,\\
        \tpplan = \diag{\aalpha^{\oone}}.
    \end{cases}
\end{equation*}
The results of Proposition~\ref{prop:gamma-delta-full-kl} can be 
adapted to this case, and yield the following formulas
\begin{subequations}
    \label{eq:gamma-pm-suot-kl}
    \begin{align}
    \displaystyle
        \gamma_-(\delta) 
        &= 
        - \frac{1}{\tau^2} \ln\left(1- \frac{\delta}{\sigma_-}\right),\\
    \displaystyle
        \gamma_+(\delta) 
        &= 
        \frac{1}{\tau^2} \ln\left(1+\frac{\delta}{\sigma_++\tpplan_{i_0,j_0}}\right),\\
    \displaystyle
    	\dmax &= 
    	\frac{1-\exp(-\tau^2\Delta)}{
    	\frac{1}{\sigma_-}+\frac{\exp(-\tau^2\Delta)}{(\sigma_++\tpplan_{i_0,j_0})}
    	 },\\
    \displaystyle
        \delta_\mathrm{eff}^+ &=
        (\sigma_++\tpplan_{i_0,j_0})
        \min_{\parent[u]\in\rightt[\tree^+]}
        \frac{\tpplan_{\parent[u]\to u}}{
        \sigma(u)+\tpplan_{\parent[u]\to u}},\\
    \displaystyle
        \delta_\mathrm{eff}^- &=
        \sigma_-
        \min_{\parent[u]\in\leftt[\tree^-]}
        \frac{\tpplan_{\parent[u]\to u}}{
        \sigma(u)+\tpplan_{\parent[u]\to u}}.
    \end{align}
\end{subequations}
In Section~\ref{sec:num-exp}, we will illustrate a possible application of 
SUOT with $\KL$ relaxation to Wasserstein gradient flows. Beforehand, we 
end this section by summarizing its results in Table~\ref{tab:recap-formules}, 
and turn to complexity analysis in the next section.
\begin{table}[ht]
    \centering
    \tiny
    \setlength{\tabcolsep}{4pt}
    \renewcommand{\arraystretch}{2.6}
    \newcommand{\Tij}{\tpplan_{i_0,j_0}}
    \newcommand{\sT}{\widetilde{\sigma}}
    \newcommand{\lP}{\lambda(\roott[\tree^+])}
    \newcommand{\lM}{\lambda(\roott[\tree^-])}
    \newcommand{\lPz}{\lambda(\roott[\tree^+])}
    \newcommand{\lMz}{\lambda(\roott[\tree^-])}
    \newcommand{\Upu}{U_{\parent[u]\to u}}
    \newcommand{\Tpu}{\tpplan_{\parent[u]\to u}}
        \begin{tabular}{p{.5cm}|>{\centering\arraybackslash}p{1cm}|>{\centering\arraybackslash}p{6.2cm}|>{\centering\arraybackslash}p{3.7cm}}
    \toprule[2pt]
    Rlx° & Quantity & UOT & SUOT \\
    \midrule
    $L^2$ & $\gamma_-(\delta)$
     & $\displaystyle \frac{\delta}{\lM}$
     & $\displaystyle \frac{\delta}{\lMz}$ \\
     & $\gamma_+(\delta)$
     & $\displaystyle \frac{\delta}{\lP}$
     & $\displaystyle \frac{\delta}{\lPz}$ \\
     & $\dmax$
     & $\displaystyle \Delta\left(\frac{1}{\lP}+\frac{1}{\lM}\right)^{-1}$
     & $\displaystyle \Delta\left(\frac{1}{\lPz}+\frac{1}{\lMz}\right)^{-1}$ \\
     & $\delta_{\mathrm{eff}}^+$
     & $\displaystyle \lP \min_{\parent[u]\in\rightt[\tree^+]}\frac{\Tpu}{\lambda(u)}$
     & $\displaystyle \lPz \min_{\parent[u]\in\rightt[\tree^+]}\frac{\Tpu}{\lambda(u)}$ \\
     & $\delta_{\mathrm{eff}}^-$
     & $\displaystyle \lM \min_{\parent[u]\in\leftt[\tree^-]}\frac{\Tpu}{\lambda(u)}$
     & $\displaystyle \lMz \min_{\parent[u]\in\leftt[\tree^-]}\frac{\Tpu}{\lambda(u)}$ \\
    \midrule
    $\KL$ & $\gamma_-(\delta)$
     & 
     $\displaystyle\frac{1}{\tau}
            \left(
            \mathrm{arcsinh}\left(
                \frac{\delta+\tpplan_{i_0,j_0}}{2\sqrt{\sigma_-\widetilde{\sigma}_-}}
                \right)-\frac{1}{2}\ln\left(\frac{\widetilde{\sigma}_-}{\sigma_-}\right)
            \right)$
     & $\displaystyle -\frac{1}{\tau^2}\ln\!\left(1-\frac{\delta}{\sigma_-}\right)$ \\
     & $\gamma_+(\delta)$
     & 
     $\displaystyle\frac{1}{\tau}
            \left(
            \mathrm{arcsinh}\left(
                \frac{\delta+\tpplan_{i_0,j_0}}{2\sqrt{\sigma_+\widetilde{\sigma}_+}}
                \right)-\frac{1}{2}\ln\left(\frac{\widetilde{\sigma}_+}{\sigma_+}\right)
            \right)$
     & $\displaystyle \frac{1}{\tau^2}\ln\!\left(1+\frac{\delta}{\sT_+}\right)$ \\
     & $\dmax$
     & 
     $\displaystyle\frac{
        	e^{\tau\Delta}\widetilde \sigma_+ \widetilde \sigma_- 
        	- e^{-\tau\Delta}\sigma_+\sigma_-
	        }{
            \sqrt{(\widetilde\sigma_++e^{-\tau\Delta}\sigma_-)
            (e^{\tau\Delta}\widetilde\sigma_-+\sigma_+)}            
            } - \tpplan_{i_0,j_0}$
    & $\displaystyle \frac{1-e^{-\tau^2\Delta}}{\frac{1}{\sigma_-}+\frac{e^{-\tau^2\Delta}}{\sT_+}}$ 
     \\
     & $\delta_{\mathrm{eff}}^+$
     & $\displaystyle -\Tij + \min_{\parent[u]\in\rightt[\tree^+]}\left(\sT_+\Upu - \frac{\sigma_+}{\Upu}\right)$
     & $\displaystyle \sT_+ \min_{\parent[u]\in\rightt[\tree^+]}\frac{\Tpu}{\sigma(u)}$ \\
     & $\delta_{\mathrm{eff}}^-$
     & $\displaystyle -\Tij + \min_{\parent[u]\in\leftt[\tree^-]}\left(\sT_-\Upu - \frac{\sigma_-}{\Upu}\right)$
     & $\displaystyle \sigma_- \min_{\parent[u]\in\leftt[\tree^-]}\frac{\Tpu}{\sigma(u)+\tpplan_{\parent[u]\to u}}$ \\
    \bottomrule
    \end{tabular}
    \caption{
    Closed-form quantities used by Algorithm~\ref{alg:adapted-ns} for the $L^2$ and $\KL$ relaxations, in the UOT and SUOT problems.}
    \label{tab:recap-formules}
\end{table}

\section{Complexity analysis}
\label{sec:complexity}

In this section, we show that Algorithm~\ref{alg:adapted-ns} enjoys complexity
results for some relaxation functions, akin to standard network simplex
techniques, assuming for simplicity that $n_1=n_2:=n$. The original 
network simplex for minimum cost flow described in Section~\ref{subsec:nsot} performs at most
$n^3\lVert\aalpha^{\bm 1}\rVert_1\lVert \costmat\rVert_\infty$ pivots, and thus
runs in time at most $n^5\lVert\aalpha^{\bm 1}\rVert_1\lVert \costmat\rVert_\infty$,
for any pivot rule~\cite{ahuja1988network}. Using the classical
Dantzig pivot rule (see Section~\ref{subsec:numpivots}) and provided that strongly
feasible spanning trees are maintained~\cite{ahuja1988network}, the method 
requires at most
$\mathcal{O}(n \lVert\aalpha^{\bm 1}\rVert_1 \log(n\lVert\aalpha^{\bm 1}\rVert_1\lVert\costmat\rVert_\infty))$
pivots, and thus runs in time
$\mathcal{O}(n^3 \lVert\aalpha^{\bm 1}\rVert_1 \log(n\lVert\aalpha^{\bm 1}\rVert_1\lVert\costmat\rVert_\infty))$.
This follows by adapting the proof of~\cite[Th.~5.3]{ahuja1988network} to the OT
setting, combining Proposition~\ref{prop:delta-precision}
and~\cite[Lem.~1.1]{ahuja1988network}.
More advanced versions of the algorithm, relying on cost scaling
techniques~\cite{orlin1997polynomial,tarjan1997dynamic}, enjoy an improved
time complexity
$\mathcal{O}\left(n^3\log(n\lVert\costmat\rVert_\infty)\log n\right)$.
For an exhaustive list of complexity bounds for minimum cost flows, we refer to
the survey of Kov\'{a}cs~\cite[Table~1]{PKovacs_2015}.

In the rest of this section, we derive worst-case time complexities for
UOT-$L^2$ and SUOT-$L^2$. We assume that Dantzig's pivot rule is used, and denote
by $(\tpplan^k,\ff^k,\gg^k)$ the iterates,
$\Delta_k := \max[\ff^k\oplus \gg^k-\costmat]\ge 0$ the maximal dual violation,
and $z_k := \En(\tpplan^k)-\En^*$ the optimality gap.
As is common in the literature, we work under the following assumption.

\begin{assumption}
    \label{assum:integral-instance}
    The distributions $\aalpha^{\bm 1}$, $\aalpha^{\bm 2}$ and the cost matrix
    $\costmat$ are integral, with $n_1=n_2=n$. The weights $\ww^{\oone}$ and
    $\ww^{\ttwo}$ are vectors of ones, and $\tau=1/q$ for some $q\in\N^*$.
    The regularization parameters are $\tau^1=\tau^2=\tau$ for UOT-$L^2$ and
    $\tau^2=\tau$ for SUOT-$L^2$.
\end{assumption}

\subsection{UOT-\texorpdfstring{$L^2$}{L2}}
Proposition~\ref{prop:uot-l2-iter-delta2} bounds the energy decrease from below 
at each iteration. Combined with Propositions~\ref{prop:delta-precision-gen}
and~\ref{prop:grid-delta}, it yields two convergence regimes, from which we
derive Theorem~\ref{th:complex-uot}.

\begin{proposition}
    \label{prop:uot-l2-iter-delta2}
    For UOT-$L^2$, under Assumption~\ref{assum:integral-instance}, each iteration
    of Algorithm~\ref{alg:adapted-ns} satisfies
    $\En(\tpplan^k)-\En(\tpplan^{k+1})\ge \tfrac{\tau}{4}\Delta_k^2$.
\end{proposition}
\begin{proof}
    Let $\tpplan^{k+1/2}$ be the plan after the push along the cycle, if any
    ($\tpplan^{k+1/2}=\tpplan^k$ otherwise). This push changes neither the
    marginals nor the dual variables, and decreases the energy by
    $\theta\Delta_k\ge 0$. Hence it suffices to bound
    $\En(\tpplan^{k+1/2})-\En(\tpplan^{k+1})$.
    Since $\En$ is quadratic, its second-order expansion at $\tpplan^{k+1}$ is exact.
    Using the first-order conditions for both plans, we get
    \begin{equation*}
        \begin{multlined}
            \En(\tpplan^{k+1/2})-\En(\tpplan^{k+1}) =
            \left(\costmat -\ff^{k+1}\oplus\gg^{k+1}\right)
            \!:\!(\tpplan^{k+1/2}-\tpplan^{k+1})\\
            + \frac{\tau}{2} \lVert\ff^k-\ff^{k+1}\rVert_2^2
            + \frac{\tau}{2} \lVert\gg^k-\gg^{k+1}\rVert_2^2.
        \end{multlined}
    \end{equation*}
    The reduced cost $\costmat -\ff^{k+1}\oplus\gg^{k+1}$ vanishes on the
    support of $\tpplan^{k+1}$, and is nonnegative on the support of
    $\tpplan^{k+1/2}$: it vanishes on the edges of the current forest, and is
    nonnegative on the edges removed during the iteration. Thus the first term is nonnegative.
    Keeping only the terms in $i_0$ and $j_0$, and using
    $(a+b)^2\le 2(a^2+b^2)$, we obtain
    \begin{align*}
        \En(\tpplan^{k})-\En(\tpplan^{k+1}) &\ge
        \frac{\tau}{2}\left(
            \left( \ff_{i_0}^k-\ff_{i_0}^{k+1}\right)^2
        + \left(\gg_{j_0}^k-\gg_{j_0}^{k+1}\right)^2\right)\\
        &\ge \frac{\tau}{4}
            \left( \ff_{i_0}^k+\gg_{j_0}^k-\ff_{i_0}^{k+1}-\gg_{j_0}^{k+1}\right)^2
        =\frac{\tau}{4}\Delta_k^2,
    \end{align*}
    since complementarity holds on $(i_0,j_0)$ at the end of the iteration.
\end{proof}

\begin{remark}
    More generally, if each $\psi_k$ is $\mu_k$-strongly convex with an
    $L_k$-Lipschitz gradient, then
    $\En(\tpplan^{k})-\En(\tpplan^{k+1})\ge\frac{\Delta_k^2}{2\left(L_1^2\mu_1^{-1}+L_2^2\mu_2^{-1}\right)}$.
\end{remark}

\begin{proposition}
    \label{prop:grid-delta}
    For UOT-$L^2$ and SUOT-$L^2$, under Assumption~\ref{assum:integral-instance},
    if $\Delta_k>0$, then $\Delta_k\ge \tfrac{1}{n^2}$.
\end{proposition}

\begin{proof}
    We drop the superscript $k$. Let $\tree$ be a tree of the current forest.
    All the mass of $\leftt[\tree]$ is sent to $\rightt[\tree]$. The
    first-order conditions and $\tau=1/q$ then imply
    $\sum_{i\in\leftt[\tree]} \ff_i - \sum_{j\in\rightt[\tree]} \gg_j=
    q\big(\lVert\aalpha^{\bm 1}\rVert_{1,\tree}-\lVert\aalpha^{\bm 2}\rVert_{1,\tree}\big)\in\ZZ$.
    For any $i'\in\leftt[\tree]$, we  write:
    \begin{equation*}
        \sum_{i\in\leftt[\tree]} \ff_i - \sum_{j\in\rightt[\tree]} \gg_j
        = \#\tree\,\ff_{i'} +\sum_{i\in\leftt[\tree]} (\ff_i-\ff_{i'})
        - \sum_{j\in\rightt[\tree]} (\gg_j+\ff_{i'}).
    \end{equation*}
    By complementarity along the paths of $\tree$, the last two sums are
    integral. Hence $\ff_i,\gg_j\in\tfrac{1}{\#\tree}\ZZ$ on $\tree$ (if $\tree$
    is a single node, first order condition directly gives an integral potential).
    Therefore, for $i\in\leftt[\tree]$ and $j\in\rightt[\tree']$,
    $\costmat_{ij}-\ff_i-\gg_j \in\tfrac{1}{\#\tree}\ZZ+ \tfrac{1}{\#\tree'}\ZZ
    \subset \tfrac{1}{\#\tree\,\#\tree'}\ZZ$. If $\tree\neq\tree'$, the two trees
    are disjoint, so $\#\tree+\#\tree'\le 2n$ and $\#\tree\,\#\tree'\le n^2$;
    if $\tree=\tree'$, the violation lies in $\tfrac{1}{\#\tree}\ZZ$ with
    $\#\tree\le 2n\le n^2$. Hence $\Delta>0$ implies $\Delta \ge \tfrac{1}{n^2}$.
    For SUOT-$L^2$, $\ff$ is free: the same argument, starting from
    $\sum_{j\in\rightt[\tree]}\gg_j\in\ZZ$, holds with $\#\tree$ replaced by
    $\#\rightt[\tree]\le n$.
\end{proof}

\begin{theorem}
    \label{th:complex-uot}
    Under Assumption~\ref{assum:integral-instance},
    Algorithm~\ref{alg:adapted-ns} solves~\eqref{eq:rot-l2} using at most
    $\mathcal{O}\big(\tfrac{n^4}{\tau}(\lVert \aalpha^{\bm 1}\rVert_1+\lVert \aalpha^{\bm 2}\rVert_1)\big)$ operations.
\end{theorem}

\begin{proof}
    Fix $\varepsilon>0$.
    By Propositions~\ref{prop:delta-precision-gen}
    and~\ref{prop:uot-l2-iter-delta2},
    $z_{k}-z_{k+1} \ge \tfrac{\tau}{4 \lVert \tpplan^*\rVert_1^2}z_k^2$,
    hence $z_{k}\le \tfrac{4 \lVert \tpplan^*\rVert_1^2}{\tau k}$, and
    $z_k\le\varepsilon$ after at most
    $\lceil \tfrac{4 \lVert \tpplan^*\rVert_1^2}{\tau \varepsilon} \rceil$ iterations.
    Then, by Propositions~\ref{prop:uot-l2-iter-delta2}
    and~\ref{prop:grid-delta}, each iteration decreases $z_k$ by at least
    $\tfrac{\tau}{4 n^4}$, so that at most
    $\lceil \tfrac{4 n^4\varepsilon}{\tau}\rceil$ additional iterations reach $z_k=0$.
    Choosing $\varepsilon = \tfrac{\lVert \tpplan^*\rVert_1}{n^2}$ yields
    $\mathcal{O}\big(\tfrac{n^2\lVert \tpplan^*\rVert_1}{\tau}\big)$ iterations.
    We conclude with $\lVert\tpplan^*\rVert_1\le\lVert \aalpha^{\bm 1}\rVert_1+\lVert \aalpha^{\bm 2}\rVert_1$
    (Proposition \ref{lem:l1-norm-solutions}) and the $\mathcal{O}(n^2)$ cost of an iteration.
\end{proof}

\subsection{SUOT-\texorpdfstring{$L^2$}{L2}}

The analysis follows the UOT-$L^2$ case, but differs in the energy decrease guarantees. 
Those yield three convergence regimes that are used to prove Theorem~\ref{th:complex-suot}.

\begin{proposition}
    \label{prop:suot-l2-iter}
    For SUOT-$L^2$, under Assumption~\ref{assum:integral-instance}, each iteration
    of Algorithm~\ref{alg:adapted-ns} satisfies
    $\En(\tpplan^k)-\En(\tpplan^{k+1})\ge \min\left(\tfrac{\tau}{4}\Delta_k^2,\tfrac{\Delta_k}{2}\right)$.
\end{proposition}

\begin{proof}
    Let $(\tpplan^\delta,\ff^\delta,\gg^\delta)$ be the iterate when a flow
    $\delta=\tpplan_{i_0j_0}$ has been pushed on $(i_0,j_0)$, including the
    push along the cycle, if any. Let
    $\Delta^\delta:=\ff_{i_0}^\delta+\gg_{j_0}^\delta-\costmat_{i_0j_0}$, so that
    $\Delta^0=\Delta_k$. By the first-order condition and since
    $\tpplan^\delta\indicm_{n_2}=\aalpha^{\bm 1}$ is fixed,
    \begin{equation*}
        \frac{\mathrm{d}\En(\tpplan^\delta)}{\mathrm{d}\delta}
        = \sum_{ij}(\costmat_{ij}-\gg_j^\delta)\frac{\mathrm{d}\tpplan_{ij}^\delta}{\mathrm{d}\delta}
        = \left(\costmat - \ff^\delta\oplus\gg^\delta\right)\!:\!\frac{\mathrm{d}\tpplan^\delta}{\mathrm{d}\delta}
        = -\Delta^\delta,
    \end{equation*}
    since $\tpplan^\delta$ only changes on $\tree_\delta^+\cup(i_0,j_0)\cup\tree_\delta^-$
    and complementarity holds on $\tree_\delta^\pm$.
    Along the cycle, $\Delta^\delta$ is constant. Afterwards, while
    $\#\rightt[\tree_\delta^-]\ge 1$, we have by \eqref{eq:gamma-delta-uot-l2}:
    \begin{equation*}
        \frac{\mathrm{d}\Delta^\delta}{\mathrm{d}\delta}
        = -\frac{1}{\tau}\left(\frac{1}{\#\rightt[\tree_\delta^+]}+\frac{1}{\#\rightt[\tree_\delta^-]}\right)
        \ge -\frac{2}{\tau},
        \quad\text{so that}\quad \Delta^\delta\ge \Delta^0-\frac{2\delta}{\tau}.
    \end{equation*}
    The flow $\delta$ increases until either complementarity holds on
    $(i_0,j_0)$, which cannot happen before $\delta=\tfrac{\tau\Delta^0}{2}$, or
    $\delta=\aalpha_{i_0}^{\bm 1}\ge 1$. In the latter case,
    $\tree_\delta^-=\{i_0\}$, and complementarity on $(i_0,j_0)$ is reached by
    decreasing $\ff_{i_0}$ alone, without changing $\tpplan^\delta$.
    In both cases, a flow at least $L:=\min(\tfrac{\tau\Delta^0}{2},1)$ is pushed,
    and since $\Delta^\delta\ge0$,
    \begin{equation*}
        \En(\tpplan^k)-\En(\tpplan^{k+1}) \ge \int_0^{L} \Big(\Delta^0-\frac{2\delta}{\tau}\Big)\mathrm{d}\delta
        = L\Delta^0-\frac{L^2}{\tau}
        \ge \begin{cases}
            \frac{\Delta^0}{2} & \text{if } \tau\Delta^0\ge 2,\\
            \frac{\tau(\Delta^0)^2}{4} & \text{otherwise}.
        \end{cases}
    \end{equation*}
    \vspace{-0.5em}
\end{proof}

\begin{theorem}
    \label{th:complex-suot}
    Under Assumption~\ref{assum:integral-instance},
    Algorithm~\ref{alg:adapted-ns} solves~\eqref{eq:srot-l2} in time
    \begin{equation*}
        \mathcal{O}\left(
            \tfrac{n^4}{\tau}\lVert \aalpha^{\bm 1}\rVert_1
            +n^2\lVert \aalpha^{\bm 1}\rVert_1
            \log\big(\tau\lVert\costmat\rVert_\infty+\lVert \aalpha^{\bm 1}\rVert_1+\lVert \aalpha^{\bm 2}\rVert_1\big)\right).
    \end{equation*}
\end{theorem}

\begin{proof}
    Here $\lVert\tpplan^*\rVert_1=\lVert\aalpha^{\bm 1}\rVert_1$, so
    Proposition~\ref{prop:delta-precision-gen} gives
    $\Delta_k\ge z_k/\lVert\aalpha^{\bm 1}\rVert_1$. Since
    $t\mapsto\min(\tfrac{\tau}{4}t^2,\tfrac{t}{2})$ is nondecreasing,
    Proposition~\ref{prop:suot-l2-iter} remains valid with $\Delta_k$ replaced by
    any lower bound. We distinguish two regimes.
    \begin{itemize}
        \item If $z_k\ge \tfrac{2\lVert\aalpha^{\bm 1}\rVert_1}{\tau}$, then
        $\tfrac{z_k}{\lVert\aalpha^{\bm 1}\rVert_1}\ge\tfrac{2}{\tau}$ and
        $z_{k}-z_{k+1}\ge \tfrac{z_k}{2\lVert\aalpha^{\bm 1}\rVert_1}$.
        Thus at most
        $\lceil 2\lVert\aalpha^{\bm 1}\rVert_1 \log\big(\tfrac{\tau z_0}{2\lVert\aalpha^{\bm 1}\rVert_1}\big)\rceil$
        iterations reach $z_k\le\tfrac{2\lVert\aalpha^{\bm 1}\rVert_1}{\tau}$.
        \item If $z_k\le \tfrac{2\lVert\aalpha^{\bm 1}\rVert_1}{\tau}$, then
        $z_k-z_{k+1}\ge \tfrac{\tau}{4}\max\big(\tfrac{z_k}{\lVert\aalpha^{\bm 1}\rVert_1},\tfrac{1}{n^2}\big)^2$,
        and the proof of Theorem~\ref{th:complex-uot} gives at most
        $\mathcal{O}(\tfrac{n^2\lVert \aalpha^{\bm 1}\rVert_1}{\tau})$
        iterations to reach optimality.
    \end{itemize}
    Finally, for any feasible $\tpplan^0$,
    $z_0\le \En(\tpplan^0)\le \lVert\costmat\rVert_\infty\lVert\aalpha^{\bm 1}\rVert_1
    + \tfrac{1}{2\tau}\big(\lVert\aalpha^{\bm 1}\rVert_1^2+\lVert\aalpha^{\bm 2}\rVert_1^2\big)$,
    and each iteration costs $\mathcal{O}(n^2)$.
\end{proof}

\begin{remark}
    Compared to the network simplex with Dantzig's pivot rule for
    OT, our complexity for UOT-$L^2$ is larger by a factor
    $\tfrac{n}{\tau \log(n\lVert\aalpha^{\bm 1}\rVert_1\lVert\costmat\rVert_\infty)}$,
    which is reasonable given the quadratic penalties. The same observation holds for
    SUOT-$L^2$, whose bound is only worse by a logarithmic factor in 
    $\lVert\costmat\rVert_\infty$.
    Theorems~\ref{th:complex-uot} and~\ref{th:complex-suot} are also
    comparable to the complexity $\mathcal{O}(n^5\log^2n)$
    recently derived by V\'egh~\cite{vegh2016} for minimum cost flow with
    quadratic costs. However, the latter is strongly polynomial, while ours are
    only pseudo-polynomial.
\end{remark}

\section{Numerical experiments}
\label{sec:num-exp}

In this section, we investigate the practical performance of our framework. 
We implemented Algorithm~\ref{alg:adapted-ns} in \python~and converted this 
code into \cpp~using Claude Opus 5.1-5.5. This code was checked by the authors 
and extensive numerical checks are provided in the repository. Rather than building a 
high-performance implementation, our aim was to estimate the possible speedup 
that one could achieve by switching from \python~to \cpp. All the
experiments have been conducted on a MacBook with chipset M4 and 24GB RAM. 
Our code is made publicly available in the GitHub repository 
{\ttfamily AdNS-UOT}\footnote{\ttfamily\url{https://github.com/gagazzotti/AdNS-UOT}}.

We first present a comparison between several pivot rules in 
Section~\ref{subsec:numpivots}. We then perform a comparison with other 
exact variants for special cases of unbalanced optimal transport in 
Section~\ref{subsec:numuot}. Finally, we investigate the interest of our 
method for computing barycenters in Section~\ref{subsec:barycenters}.

\subsection{Pivot rules for adapted network simplex}
\label{subsec:numpivots}

Selecting a pivot edge is a crucial ingredient for practical performance of 
the simplex method, and numerous techniques have been proposed in the 
literature~\cite[Sec.11]{nf1993orlin}. We consider the four rules 
described below (only Dantzig's rule is covered by
the complexity guarantees of Theorems~\ref{th:complex-uot}
and~\ref{th:complex-suot}).

\begin{itemize}
    \item \textit{Dantzig's pivot rule} selects the edge corresponding to the 
    largest dual violation, \textit{i.e.}, 
    $\argmax_{(i_0,j_0)\in V\times V'} \ff_{i_0}+\gg_{j_0}-\costmat_{i_0,j_0}$. 
    This method generally reduces the number of overall iterates but requires 
    to examine all the $n_1n_2$ values to determine the largest 
    violation~\cite[sec. 11]{nf1993orlin}, which is expensive for large 
    instances. Theorem~\ref{th:complex-uot} was established using this rule.
    \item The \textit{candidate list} pivot rule collects a list of $k$ 
    violating edges and chooses those in order for the next $k$ iterations 
    provided they remain violating edges (otherwise a new list of $k$ violating 
    edges is selected). This procedure does not require examining the $n_1 n_2$ 
    dual constraints at every iteration but tends to produce more iterations 
    overall. In our implementation, we choose the top-$k$ dual violations as 
    a candidate list.
    \item The \textit{block-search} rule 
    (a.k.a.~\textit{first eligible 
    arc} rule) consists in finding the maximal dual violation on  subsets of the edges. 
    The set of edges is divided into blocks, a violating edge is sought blockwise. When it is found inside a 
    block, this block is considered first at the next iteration. This 
    procedure corresponds to that implemented in the \cpp~LEMON 
    library~\cite{lemon}. 
    \item The \textit{scaling method} seeks an edge corresponding to a dual 
    violation of at least $\Delta>0$ as long as such an edge exists. The 
    procedure then iteratively shrinks $\Delta$ by a factor of $2$ and repeats 
    the process until reaching near machine precision. This scaling approach 
    can be combined with the block-search rule, which is the choice we 
    adopted in our experiments. 
\end{itemize}  

To compare the four pivot rules above, we consider randomly generated 
instances of UOT-$L^2$ with $n_1=n_2=n$ and $\tau^1=\tau^2=1e-3$. In those 
random instances, the coefficients of $\costmat$, $\aalpha^{\oone}$ and 
$\aalpha^{\ttwo}$ are drawn i.i.d. uniformly at random in $[0,1]$.
Figure \ref{fig:pivot-rules} shows the average computational time obtained 
over $5$ randomly generated instances for each dimension. As expected, 
Dantzig's pivot rule performs worst for large values of $n$. The scaling 
method outperforms the three other rules by at least a factor of 2 on both 
\python~and \cpp~implementations, while the computational time grows at 
similar rates for the candidate list, block-search and scaling variants.
Overall, this experiment motivates our choice of the scaling method as our 
pivot rule in subsequent tests. Note that the \cpp~implementation is about ten 
times faster than the \python~one, which illustrates the potential of a 
high-performance implementation of our algorithm.

\begin{figure}
    \centering
    \includegraphics[width=0.6\linewidth]{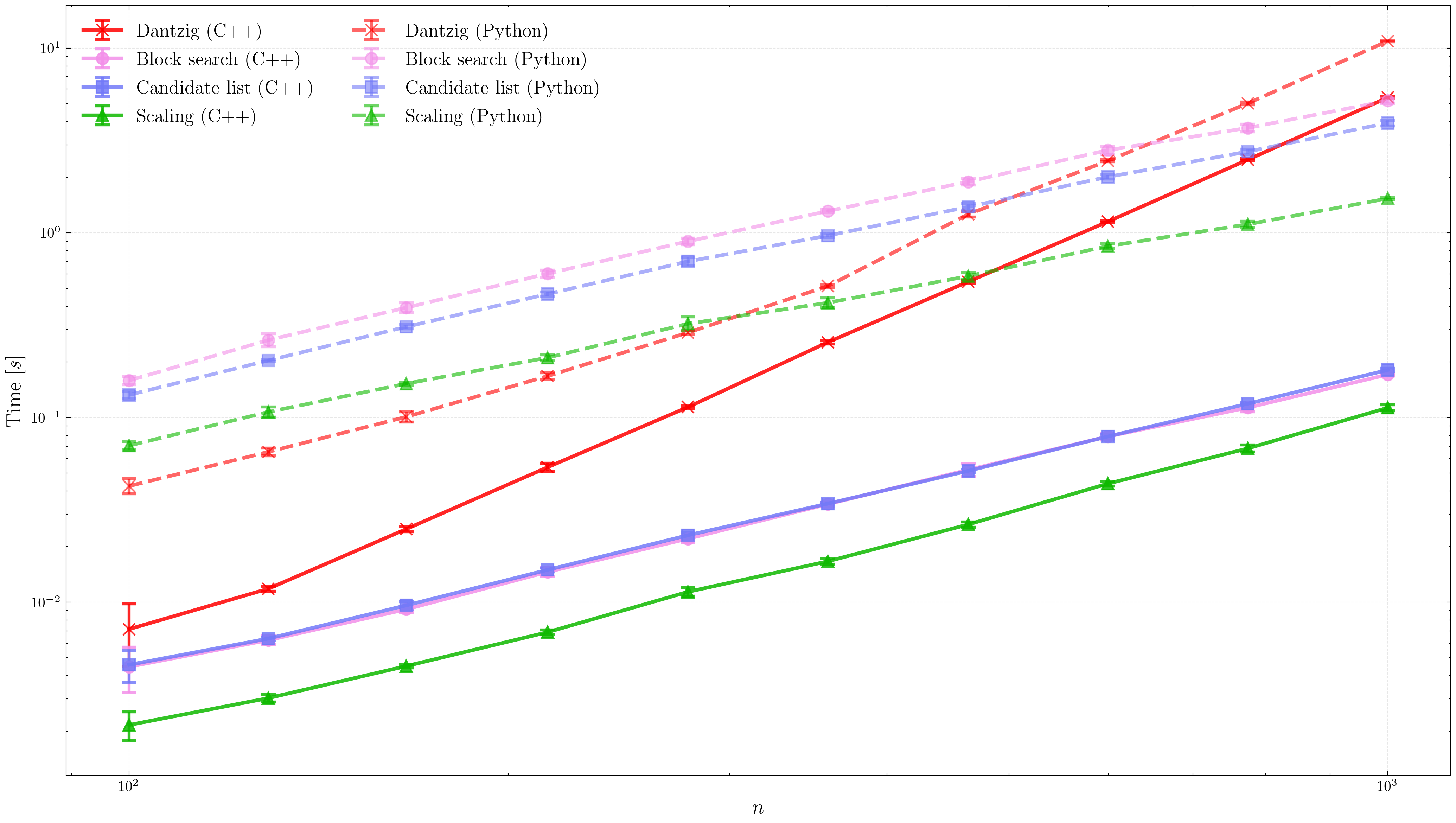}
    \caption{Comparison of the different pivot rules (\python~and \cpp~
    implementations)}
    \label{fig:pivot-rules}
\end{figure}

\subsection{Comparisons on UOT and SUOT problems}
\label{subsec:numuot}

We now present a comparison with other algorithms for solving the exact 
problem~\eqref{eq:uot}, i.e. we discard methods that add an entropic 
regularization term to the objective because those do not enforce sparsity of 
the transport plan. For the UOT-$L^2$ problem~\eqref{eq:rot-l2}, we compare 
Algorithm~\ref{alg:adapted-ns} with \celer~\cite{massias2018celer} (applied to 
the LASSO reformulation of the problem) and the general-purpose LP solver 
\mosek~\cite{mosek} obtained from the {\ttfamily cvxpy} 
interface~\cite{diamond2016cvxpy}. We also included the regularization path 
method of Chapel et al~\cite{chapel2021unbalanced}, which tracks the evolution 
of the support through a change of regularization. When we consider the 
semi-unbalanced version of~\eqref{eq:rot-l2}, we only compare with \mosek~and 
regularization path, as \celer~is no longer applicable. Finally, for both 
unbalanced and semi-unbalanced versions of optimal transport with KL 
relaxation~\eqref{eq:rot-kl}, denoted by UOT-KL and SUOT-KL, respectively, the 
regularization path is no longer applicable. Therefore, we only compare with 
\mosek~launched so as to reach the same precision as that obtained by 
Algorithm~\ref{alg:adapted-ns}.

\begin{table}
    \centering
    \tiny
    \begin{tabular}{p{1.1cm} p{1.4cm} p{1.8cm} p{1.8cm} p{1.8cm} p{1.8cm}}
    \toprule[2pt]
    Problem & Solver & \multicolumn{2}{c}{$n=500$ and $\tau=$1e-1} & \multicolumn{2}{c}{$n=500$ and $\tau=$1e-3} \\
    \midrule
     &  & \centering time [$s$] &  $|\mathcal{E}(\mathbf P)-\mathcal{E}(\mathbf P^*)|$ & \centering time [$s$] &  $|\mathcal{E}(\mathbf P)-\mathcal{E}(\mathbf P^*)|$ \\
    \midrule
    SUOT-KL & \mosek & 1.1e+0$\pm$2.7e-2 & 6.0e-8$\pm$1.3e-8 & 1.0e+1$\pm$8.2e-1 & 2.3e-5$\pm$9.1e-7 \\
     & AdNS-\python & 3.7e+0$\pm$1.1e-1 & - & 1.2e+1$\pm$9.6e-2 & - \\
     & AdNS-\cpp & 1.3e-1$\pm$3.6e-3 & - & 5.0e-1$\pm$6.2e-3 & - \\
    \midrule
    SUOT-L2 & Reg. Path & 6.5e+0$\pm$1.4e-1 & - & 1.0e+1$\pm$1.4e-1 & - \\
     & \mosek & 9.9e-1$\pm$3.6e-3 & 7.3e-7$\pm$1.6e-7 & 9.1e-1$\pm$6.0e-3 & 8.7e-5$\pm$1.1e-5 \\
     & AdNS-\python & 9.5e-2$\pm$1.7e-3 & - & 8.3e-1$\pm$1.9e-2 & - \\
     & AdNS-\cpp & 1.0e-2$\pm$5.1e-4 & - & 3.8e-2$\pm$5.9e-4 & - \\
    \midrule
    UOT-KL & \mosek & 1.2e+0$\pm$1.2e-2 & 2.1e-7$\pm$2.0e-8 & 7.4e+0$\pm$1.1e+0 & 3.3e-5$\pm$2.2e-6 \\
     & AdNS-\python & 3.1e+0$\pm$2.9e-2 & - & 1.6e+1$\pm$1.8e-1 & - \\
     & AdNS-\cpp & 1.1e-1$\pm$7.1e-4 & - & 5.4e-1$\pm$5.6e-3 & - \\
    \midrule
    UOT-L2 & celer & 8.9e-2$\pm$3.3e-3 & 2.0e-10$\pm$7.2e-11 & 4.8e+0$\pm$2.6e-1 & 2.0e-7$\pm$8.4e-8 \\
     & Reg. Path & 2.7e+0$\pm$1.1e-1 & - & 5.2e+0$\pm$1.1e-1 & - \\
     & \mosek & 1.1e+0$\pm$3.9e-2 & 9.0e-6$\pm$2.3e-6 & 1.0e+0$\pm$6.3e-3 & 2.0e-5$\pm$4.1e-6 \\
     & AdNS-\python & 6.6e-2$\pm$4.8e-4 & - & 6.2e-1$\pm$1.2e-2 & - \\
     & AdNS-\cpp & 7.8e-3$\pm$1.4e-4 & - & 3.0e-2$\pm$1.9e-4 & - \\
    \end{tabular}

    \vspace{0em}

    \begin{tabular}{p{1.1cm} p{1.4cm} p{1.8cm} p{1.8cm} p{1.8cm} p{1.8cm}}
    \toprule[2pt]
    Problem & Solver & \multicolumn{2}{c}{$n=1000$ and $\tau=$1e-1} & \multicolumn{2}{c}{$n=1000$ and $\tau=$1e-3} \\
    \midrule
    SUOT-KL & \mosek & 5.5e+0$\pm$2.3e-1 & 7.5e-8$\pm$6.1e-9 & 4.0e+1$\pm$3.9e+0 & 2.0e-5$\pm$3.0e-6 \\
     & AdNS-\python & 1.5e+1$\pm$2.0e-1 & - & 5.8e+1$\pm$2.2e-1 & - \\
     & AdNS-\cpp & 5.5e-1$\pm$5.5e-3 & - & 2.4e+0$\pm$2.0e-2 & - \\
    \midrule
    SUOT-L2 & Reg. Path & 5.4e+1$\pm$3.4e-1 & - & 8.5e+1$\pm$1.5e+0 & - \\
     & \mosek & 4.4e+0$\pm$1.6e-2 & 3.9e-7$\pm$1.5e-7 & 4.4e+0$\pm$1.8e-2 & 1.0e-4$\pm$6.9e-6 \\
     & AdNS-\python & 2.6e-1$\pm$2.8e-3 & - & 2.0e+0$\pm$2.5e-2 & - \\
     & AdNS-\cpp & 4.5e-2$\pm$6.1e-4 & - & 1.4e-1$\pm$2.1e-3 & - \\
    \midrule
    UOT-KL & \mosek & 6.0e+0$\pm$7.7e-2 & 1.9e-7$\pm$1.4e-8 & 2.9e+1$\pm$4.2e+0 & 4.4e-5$\pm$2.4e-6 \\
     & AdNS-\python & 1.1e+1$\pm$1.2e-1 & - & 7.3e+1$\pm$6.8e-1 & - \\
     & AdNS-\cpp & 4.5e-1$\pm$6.7e-3 & - & 2.5e+0$\pm$1.7e-2 & - \\
    \midrule
    UOT-L2 & \celer & 3.8e-1$\pm$5.9e-3 & 1.5e-10$\pm$5.0e-11 & 2.0e+1$\pm$2.1e+0 & 1.3e-7$\pm$8.7e-8 \\
     & Reg. Path & 2.2e+1$\pm$4.2e-1 & - & 4.4e+1$\pm$8.5e-1 & - \\
     & \mosek & 4.9e+0$\pm$3.9e-2 & 5.7e-5$\pm$2.7e-5 & 4.5e+0$\pm$8.6e-2 & 1.0e-3$\pm$2.2e-4 \\
     & AdNS-\python & 2.0e-1$\pm$7.4e-3 & - & 1.5e+0$\pm$2.6e-2 & - \\
     & AdNS-\cpp & 3.4e-2$\pm$2.2e-4 & - & 1.1e-1$\pm$1.8e-3 & - \\
    \bottomrule
    \end{tabular}
    \caption{Comparison of Algorithm~\ref{alg:adapted-ns} (AdNS), regularization path (Reg. Path), 
    \mosek and \celer on four unbalanced formulations.}
    \label{tab:results}
\end{table}

Table~\ref{tab:results} presents the comparison for several values of 
$n=n_1=n_2$ and $\tau$, where $\tau=\tau^1=\tau^2$ for unbalanced 
formulations and $\tau=\tau^2>\tau^1=0$ in semi-unbalanced 
formulations. For each pair $(\tau,n)$, we randomly generated 
$5$ instances per formulation following the same procedure as in 
Section~\ref{subsec:numpivots} and we report statistics over those 
problems, both in terms of computational time and gap to optimality. 
Our \cpp~ implementation is about ten times faster than 
\mosek~ and provides an exact solution, unlike that of \mosek. The \python~
implementation of Algorithm~\ref{alg:adapted-ns} is about 10 to 100 times 
faster than regularization path for both formulations using $L^2$ relaxation.
Overall, our experiments advocate for using our method on $L^2$ and KL 
relaxed formulations of optimal transport. We provide further motivating 
experiments below.

\paragraph{Solution sparsity}
Our main motivation for developing Algorithm~\ref{alg:adapted-ns} was to 
maintain at most $n_1+n_2-1$ nonzero elements of the primal variable at each 
iteration, as in network simplex implementations. Even though sparsity is also 
guaranteed in the regularization path algorithm \cite{chapel2021unbalanced}, 
this method is restricted to UOT-$L^2$ and SUOT-$L^2$ problems. 
Figure~\ref{fig:sparsity} displays the sparsity of the solutions obtained by 
the regularization path~\cite{chapel2021unbalanced} and \mosek~compared to 
that of Algorithm~\ref{alg:adapted-ns}. We also compare with two Sinkhorn 
algorithms~\cite{GPeyre_MCuturi_2019} with different entropic 
regularization parameters. We note that both regularization path updates and 
\mosek ~give denser approximations of the solution than our algorithm. For 
Sinkhorn variants, the sparsity increases as the entropic regularization 
parameter decreases, as expected.

\begin{figure}
    \centering
    \includegraphics[width=0.6\linewidth]{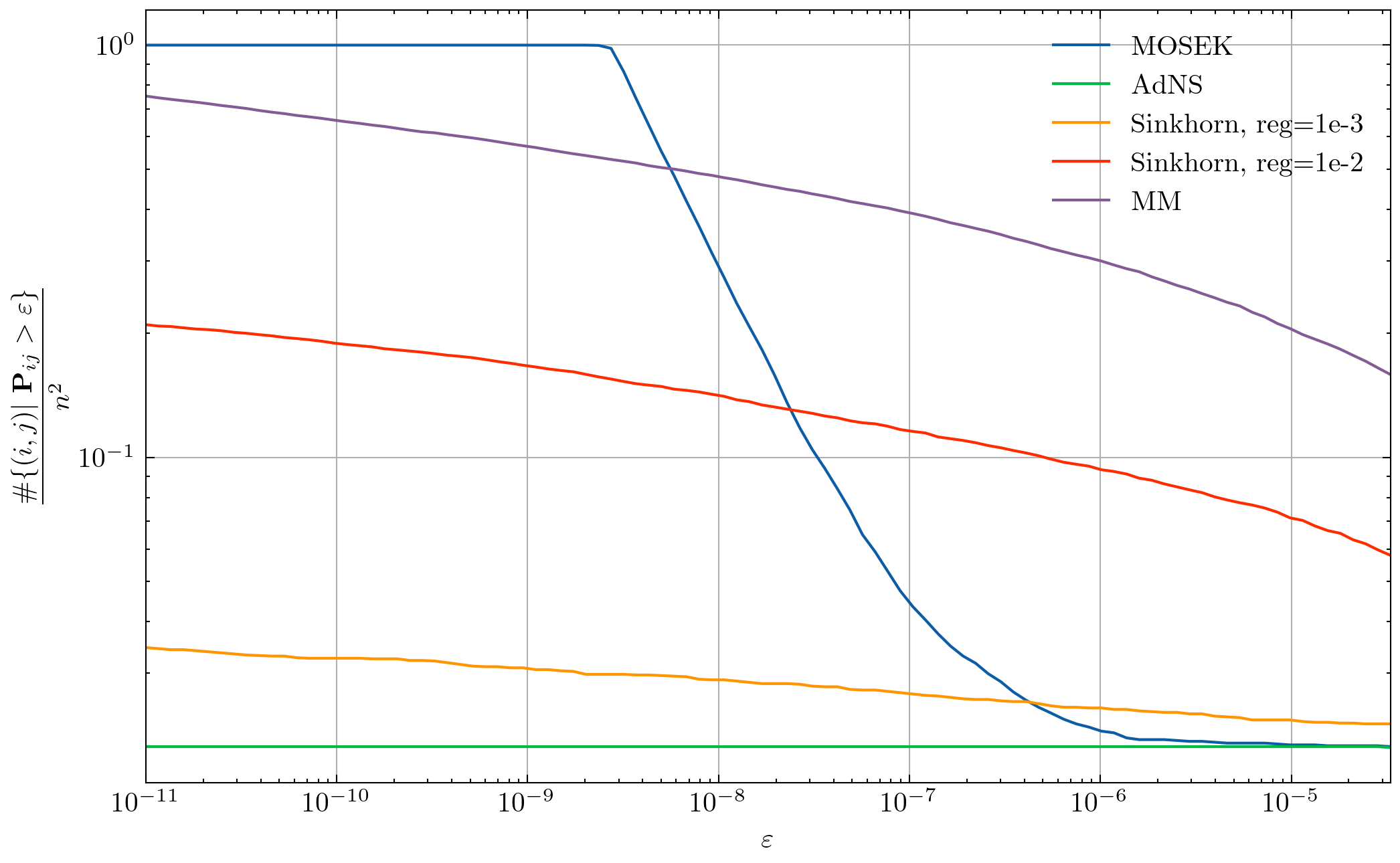}
    \caption{Sparsity of the solutions obtained with various algorithms with 
    $n=100$ and $\tau=\text{1e-3}$. For any $\varepsilon$ on the $x$-axis, the 
    $y$-axis shows the number of coordinates of the solution greater than 
    $\varepsilon$. The largest value of $\varepsilon$ corresponds to the 
    minimum nonzero coefficient of the solution computed by 
    Algorithm~\ref{alg:adapted-ns}. For Sinkhorn methods, \emph{reg} 
    denotes the entropic regularization parameter.}
    \label{fig:sparsity}
\end{figure}

\paragraph{Comparison with network simplex} 
We also compare Algorithm~\ref{alg:adapted-ns} on problem~\eqref{eq:rot-l2} 
on a random instance (generated as in Section~\ref{subsec:numpivots})
with $n=25000$ and several values of $\tau=\tau^1=\tau^2$ to the standard 
network simplex method applied to the corresponding optimal transport 
formulation~\eqref{eq:ot}. For the latter, we used the implementation 
from the {\ttfamily POT} library~\cite{pot}. In Figure~\ref{fig:uot-vs-ot}, 
we plot the computational times of both methods. For $\tau \ge \text{1e-5}$, 
we observe that Algorithm~\ref{alg:adapted-ns} is faster than the network 
simplex to reach a precision of $10^{-5}$ (in $L^1$ norm) on the marginals. 
These results encourage using unbalanced formulations as proxies for optimal 
transport to benefit from fast solving times.

\begin{figure}
    \centering
    \includegraphics[width=0.6\linewidth]{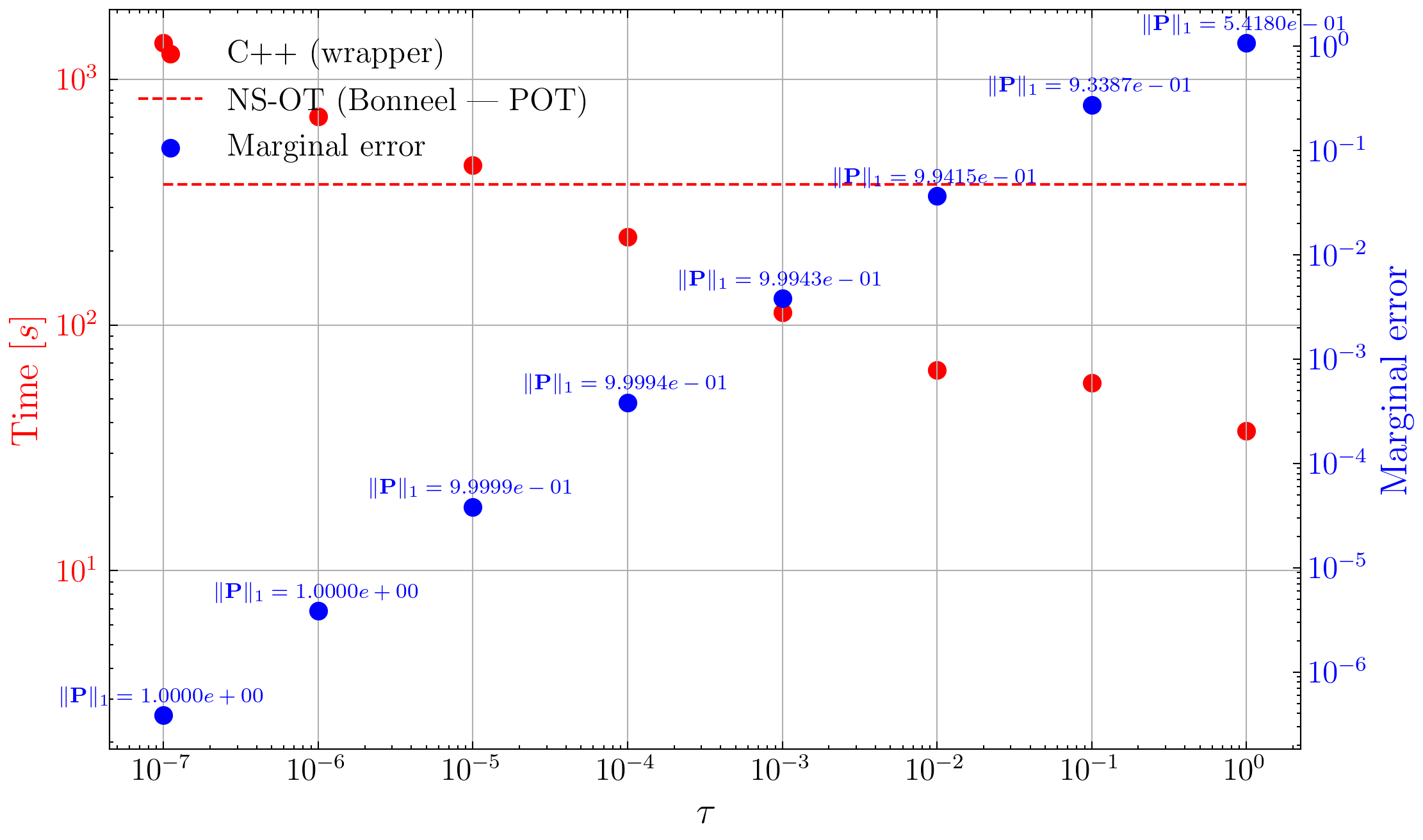}
    \caption{Computational time of Algorithm~\ref{alg:adapted-ns} for 
    problem~\eqref{eq:rot-l2} with several values of $\tau$ (circles) 
    and Network simplex for~\eqref{eq:uot} (dashed line).}
    \label{fig:uot-vs-ot}
\end{figure}

\paragraph{Extension to other relaxations}
As a final experiment, we apply our algorithmic framework to another 
choice of relaxation function based on $L^p$ norms with even 
$p$~\cite{lee2019parallel}. Consider~\eqref{eq:rot} with 
$n_1=n_2=n$ and
$\psi_k(\bm x):=(p\tau)^{-1}\|\bm x-\aalpha^{\bm k}\|_p^p$ for 
$k \in \{1,2\}$, where $\aalpha^{\oone}$ and $\aalpha^{\ttwo}$ 
are two measures that need not have identical mass. We can define a 
feasible initialization by 
\[
	\tpplan=\mathbf{0}_{n\times n}, 
	\quad
	\ff=-\tfrac{1}{\tau}\left(-\aalpha^{\oone}\right)^{p-1}, 
	\quad
	\gg=-\tfrac{1}{\tau}\left(-\aalpha^{\ttwo}\right)^{p-1},
\]
however the equations~\eqref{eq:delta-max}, \eqref{eq:delta-eff-tm} 
and \eqref{eq:delta-eff-tp} no longer admit closed-form 
solutions. We thus solve these 
nonlinear equations 
using Brent's method~\cite{brent1973algorithms} implemented 
in {\ttfamily Scipy} version 1.16.
We compare Algorithm~\ref{alg:adapted-ns} to the {\ttfamily clarabel} 
solver~\cite{Clarabel_2024} (available through {\ttfamily CVXPY}) and 
\mosek. On an instance with $p=4$, $n=500$, $\tau=$1e-4, and $\aalpha^{\oone}$, 
$\aalpha^{\ttwo}$ and $\costmat$ randomly generated as in 
Section~\ref{subsec:numpivots}, Algorithm~\ref{alg:adapted-ns} finds a 
solution in 0.16 seconds while the other solvers take 5.21 seconds and 1.01 
seconds, respectively.

\subsection{Applications to computing barycenters}
\label{subsec:barycenters}

\begin{figure}[htb]
  \centering
  \begin{subfigure}{0.95\textwidth}
    \centering
    \includegraphics[width=\linewidth]{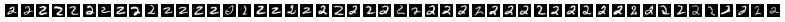}
  \end{subfigure}
  \begin{subfigure}{0.95\textwidth}
    \centering
    \includegraphics[width=\linewidth]{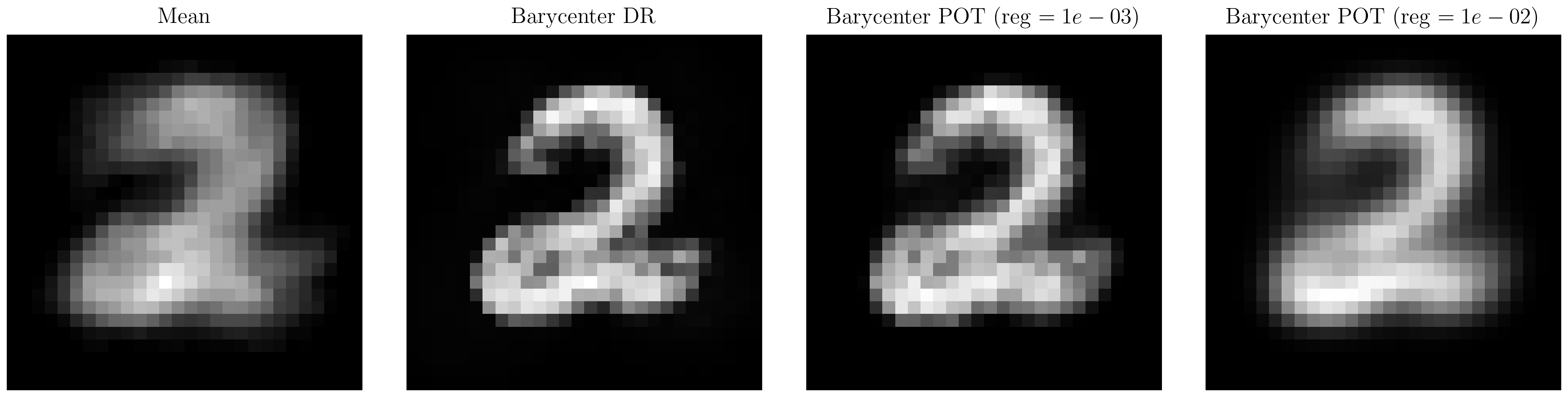}
  \end{subfigure}
  \caption{Computation of Wasserstein barycenter of 50 MNIST digits using 
  Algorithm~\ref{alg:adapted-ns} to compute proximal steps 
  in~\eqref{eq:driterates} with $\tau=$1e-1.}
  \label{fig:barycenter}
\end{figure}

Consider $m$ discrete probability distributions 
$(\mmu_k)_{k\in\IntInterval{1}{m}}$ of size $n$, defined on the same discrete 
support $(\bm{x}_i)_{i\in\IntInterval{1}{n}}$ and their associated cost matrix 
$\costmat = 
\left(\|\bm{x}_i-\bm{x}_j\|_2^2\right)_{(i,j)\in\IntInterval{1}{n}^2}$. The 
discrete Wasserstein barycenter problem~\cite{carlier2011wb} 
consists in solving
\begin{equation}
\label{eq:w2barycenter}
    \min_{\nnu\in\R^n}
    \sum_{k=1}^m \omega_k W_2^2(\mmu_k,\nnu),
\end{equation}
where $\mathbf{\omega}$ is a vector of weights, and $W_2^2(\mmu,\nnu)$ 
corresponds to the optimal value of the optimal transport 
problem~\eqref{eq:ot}. We explore several ways of computing this quantity as 
well as an alternate formulation.

\paragraph{Computing a barycenter using Douglas-Rachford and SUOT}
We consider an unweighted version of~\eqref{eq:w2barycenter} corresponding to all 
weights being equal. Problem~\eqref{eq:w2barycenter} can then be reformulated as
\begin{equation}
\label{eq:w2barycenteradmm}
	\min_{(\nnu_1,...,\nnu_m)\in\R^{n\times m}}
        \sum_{i=1}^m W_2^2(\mmu_i,\nnu_i)
        + \mathbbm{1}_{\mathcal{C}_m}(\nnu_1,\dots,\nnu_m).
\end{equation}
where $\mathcal{C}_m = 
\{(\bm x_1,\dots,\bm x_m)\in\R^{n\times m}|\bm x_1=\dots=\bm x_m\}$.
Problem~\eqref{eq:w2barycenteradmm} can then  be tackled using alternating 
methods, such as the Douglas-Rachford (DR) 
splitting~\cite{AChambolle_JPContreras_2022, LionsMercier}, leading to the 
following recursion:
\begin{equation}
\label{eq:driterates}
    \begin{cases}
        \displaystyle
        \bar{\bm{u}}^{(k+1)} 
        = 
        \frac{1}{m}\sum_{t=1}^m \bm{v}_t^{(k)}\\
        \forall t \in \IntInterval{1}{m},~
        \bm{p}_t^{(k+1)} 
        = 
        \mathrm{prox}_{\mathrm{OT}(\mmu_t,\cdot|\costmat)}^\tau
        \left(2\bar{\bm{u}}^{(k+1)} - \bm{v}_t^{(k)}\right)\\
        \forall t \in \IntInterval{1}{m},~
        \bm{v}_t^{(k+1)} 
        = 
        \bm{v}_t^{(k)}+\bm{p}_t^{(k+1)}-\bar{\bm u}^{(k+1)}.
    \end{cases}
\end{equation}
Although the method~\eqref{eq:driterates} converges for any positive
$\tau$~(\cite{LionsMercier,eckstein1992douglas}), the actual value of 
$\tau$ has a significant practical impact. We implement the proximal 
update of $\bm{p}_t^{(k+1)}$ by solving a semi-unbalanced optimal 
transport problem of the form~\eqref{eq:srot-l2}, and apply it to 
compute the barycenter of 50 images from the MNIST 
database~\cite{YLeCun_LBottou_YBengio_PHaffner_1998}. 
Figure~\ref{fig:barycenter} compares our approach with an iterative 
Bregman projection~\cite{benamou2015iterative} available within the 
POT library, that relies on entropic regularization. Our method yields 
a sharper image compared to the output of iterative Bregman projection 
for both values of regularization. Note that the iterative Bregman 
projection suffers from numerical issues for $\mathrm{reg}=10^{-4}$, 
and thus outputs a poor quality barycenter.

\begin{figure}[h]
    \centering
    \begin{subfigure}[b]{0.24\textwidth}
        \centering
        \includegraphics[width=\textwidth]{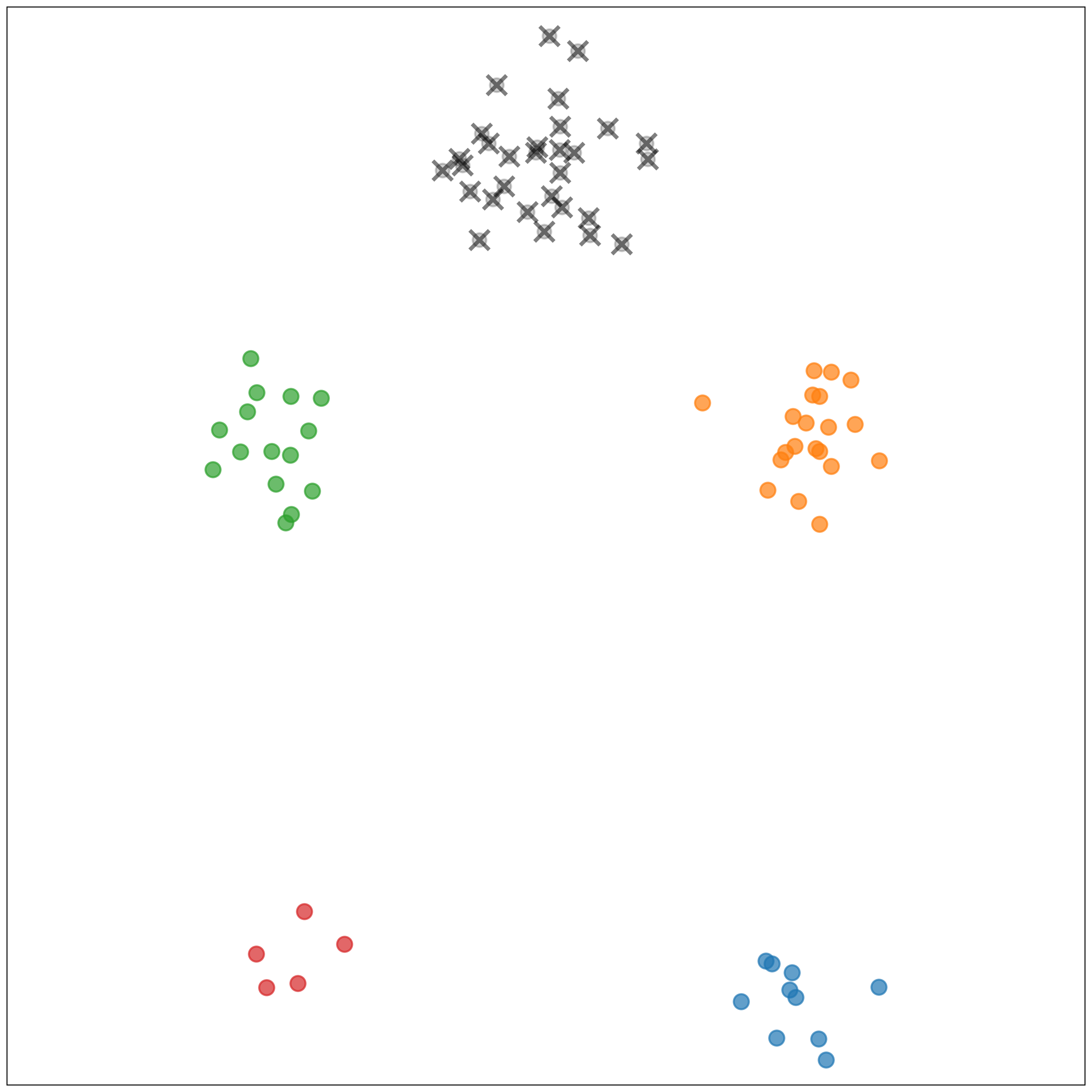}
        \caption{$\tau = 1$e-2}
    \end{subfigure}
    \begin{subfigure}[b]{0.24\textwidth}
        \centering
        \includegraphics[width=\textwidth]{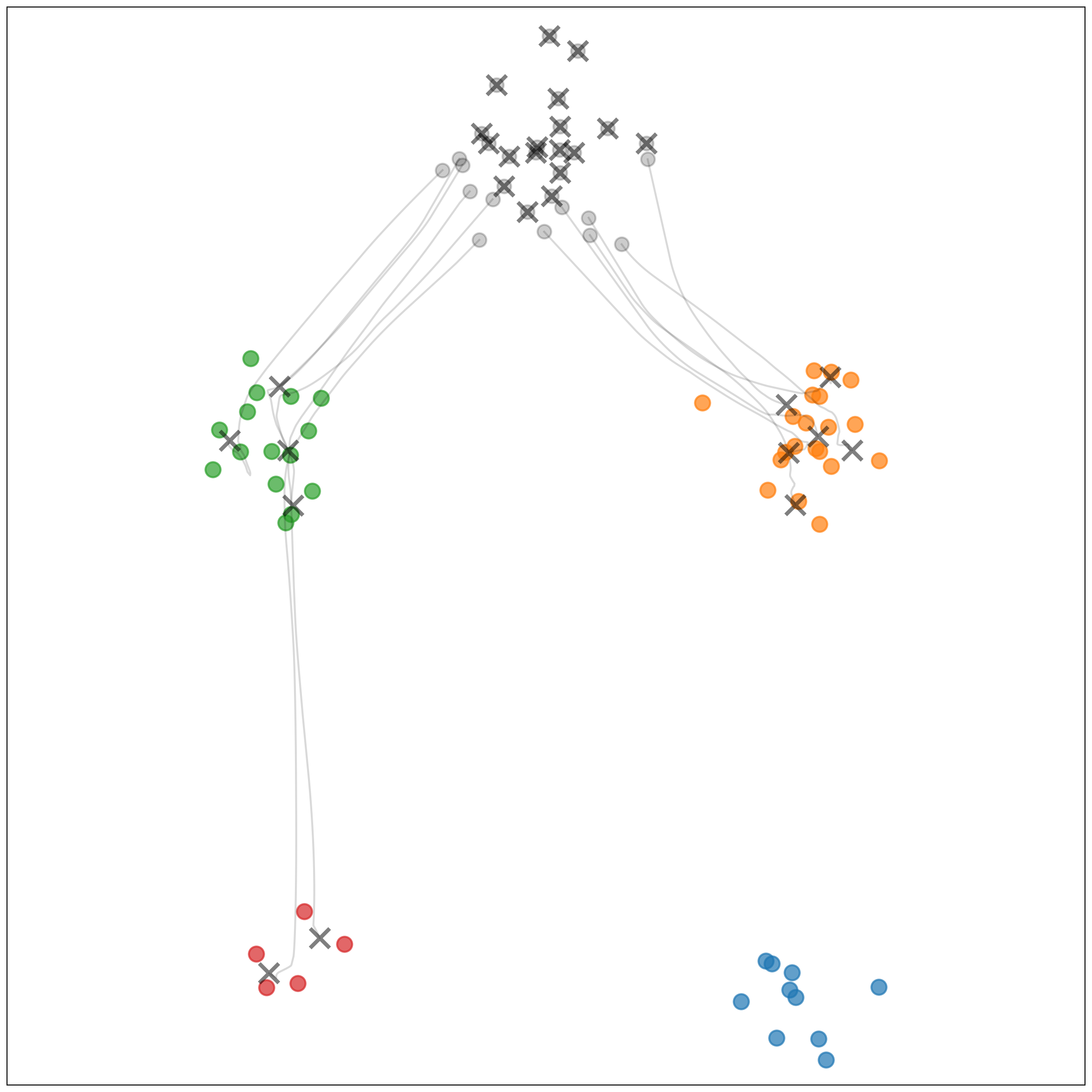}
        \caption{$\tau = 1$e-2.5}
    \end{subfigure}
    \begin{subfigure}[b]{0.24\textwidth}
        \centering
        \includegraphics[width=\textwidth]{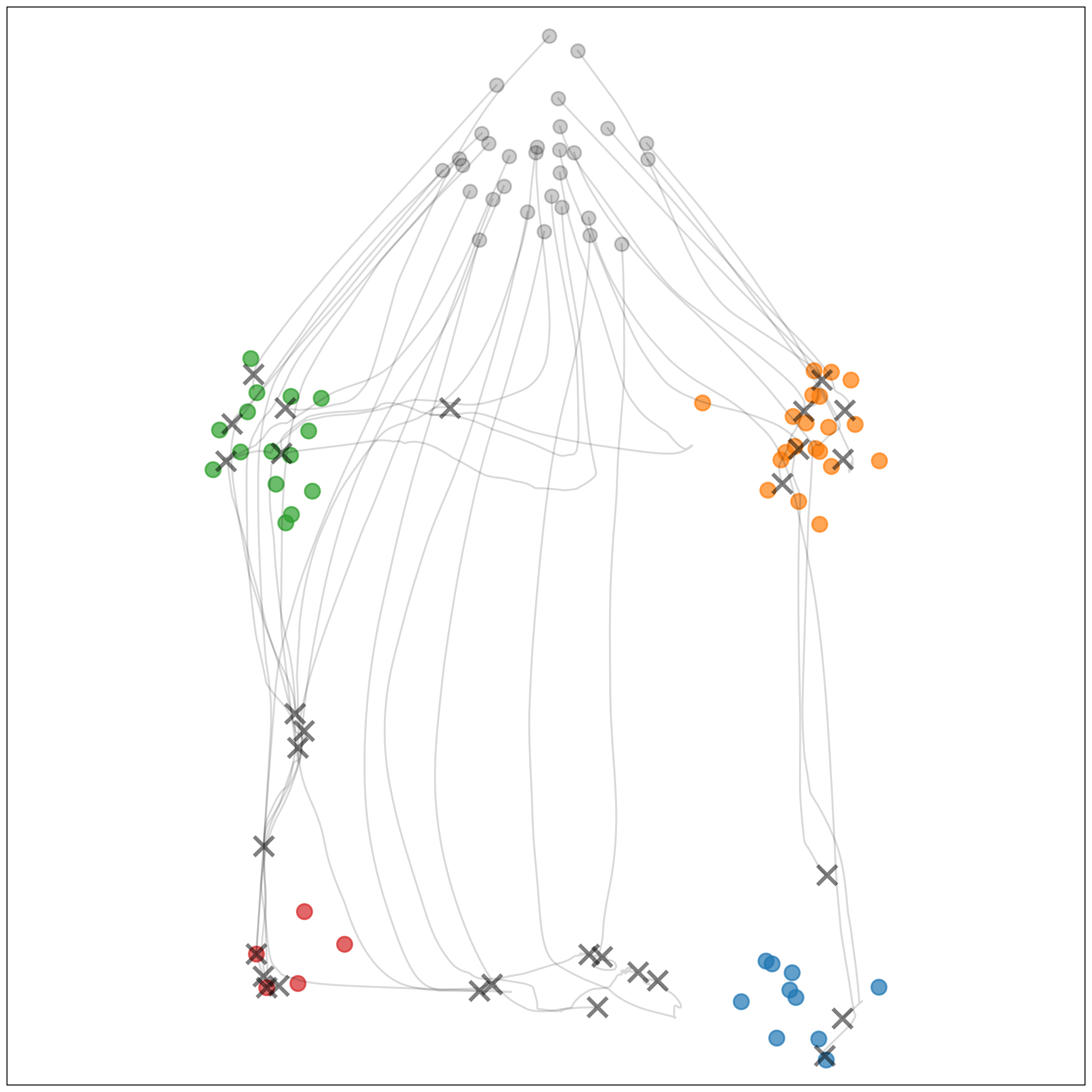}
        \caption{$\tau = 1$e-3}
    \end{subfigure}
    \begin{subfigure}[b]{0.24\textwidth}
        \centering
        \includegraphics[width=\textwidth]{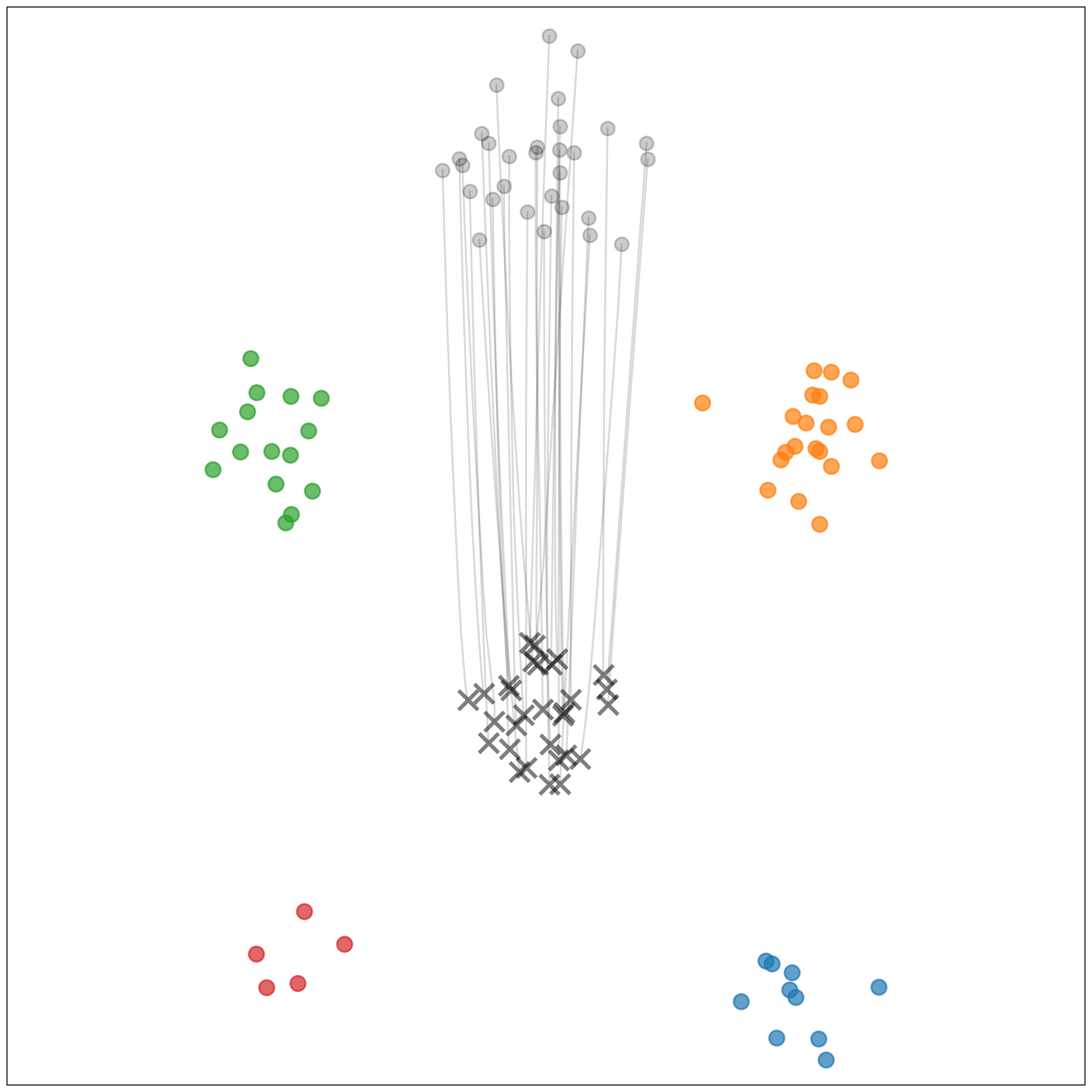}
        \caption{$\tau = 1$e-4}
    \end{subfigure}
    
    \caption{Gradient descent trajectories for problem~\eqref{eq:fsbarycenterpb} 
    (grey dots are initial points, crosses correspond to last iterates). 
    Each color corresponds to a different measure $\bm{\nu}_\ell$.}
    \label{fig:fs-barycenter}
\end{figure}
\paragraph{Free-support barycenter problem}
We now consider a variant of~\eqref{eq:w2barycenter} where all weights are 
equal, and the support of the barycenter is less constrained by that of the 
input measures. Let $\bm{\nu}_1,\dots,\bm{\nu}_m$ be $m$ discrete probability 
measures on $\R^d$ of the form $\bm{\nu}_\ell 
= \tfrac{1}{n_{\ell}}\sum_{k=1}^{n_\ell}\delta_{\bm{x}_k^{ \ell}}$ where 
$\bm{x}_k^{\ell} \in \R^d$ are discrete supports that differ for each 
measure. The free support barycenter problem (in its unweighted version) 
consists in finding an atomic distribution with $n$ Dirac masses by solving 
\begin{equation}
\label{eq:fsbarycenterpb}
    \min_{(\bm{x}_1,...,\bm{x}_n)\in(\R^{d})^n}
    \sum_{\ell=1}^m \text{UOT-}L^2
    \left(\bm{\nu}_\ell,\frac{1}{n}\sum_{k=1}^n\delta_{\bm{x}_k};
    \costmat_\ell,\tau\right),
\end{equation}
where $\text{UOT-}L^2(\aalpha^{\oone},\aalpha^{\ttwo};\costmat,\tau)$ is 
the optimal value of problem~\eqref{eq:rot-l2} with weight vectors 
$\ww^1=\ww^2=\mathbf{1}_n$ and $\tau^1=\tau^2=\tau$. The cost matrices 
$\{\costmat_{\ell}\}_{\ell=1,\dots,m}$ are defined by
$(\costmat_\ell)_{i,j} := \lVert\bm{x}_i^{\ell}-\bm{x}_j\rVert_2^2$. 
Problem~\eqref{eq:fsbarycenterpb} uses the output of unbalanced optimal 
transport as a proxy for the Wasserstein ($L^2$) distance, that would 
require to solve an optimal transport formulation instead.

Problem~\eqref{eq:fsbarycenterpb} can be solved using first-order algorithms 
since it is possible to differentiate through the unbalanced optimal 
transport solve. Indeed, for every $\ell=1,\dots,m$, we have
\begin{equation*}
    \partial_{\bm{x}_p} \text{UOT-}L^2
    \left(\bm{\nu}_\ell,\frac{1}{n}\sum_{k=1}^n\delta_{\bm{x}_k};
    \costmat_\ell,\tau\right) 
    =
    2\sum_{i=1}^{n_\ell}\tpplan_{i,p}^{\ell}
    (\bm{x}_p-\bm{x}_i^\ell)
\end{equation*}
where $\tpplan^\ell$ is a solution of the corresponding UOT-$L^2$ problem. 
Figure \ref{fig:fs-barycenter} displays the output of gradient descent 
applied to several instances of~\eqref{eq:fsbarycenterpb} (differing by 
$\tau$) for $10^5$ iterations using a stepsize of $0.2$ and a random 
Gaussian distribution as initial point. For $\tau=$1e-4, we obtain a result 
similar to what an optimal transport formulation would give, i.e. a Wasserstein 
barycenter of the $m$ measures. For larger values of $\tau$, however, the 
barycenter can cover the input measures or possibly none of them ($\tau=$1e-2), 
giving rise to a constant barycenter in the limit.

\subsection{Applications to the JKO scheme}
\label{subsec:jko}
\begin{figure}[ht]
  \centering
  \begin{subfigure}{.24\textwidth}
    \centering
    \includegraphics[width=\linewidth]{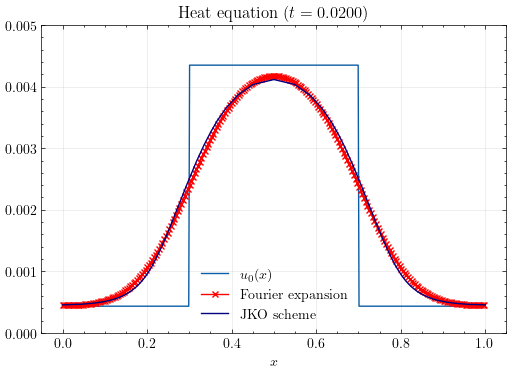}
  \end{subfigure}
  \begin{subfigure}{.24\textwidth}
    \centering
    \includegraphics[width=\linewidth]{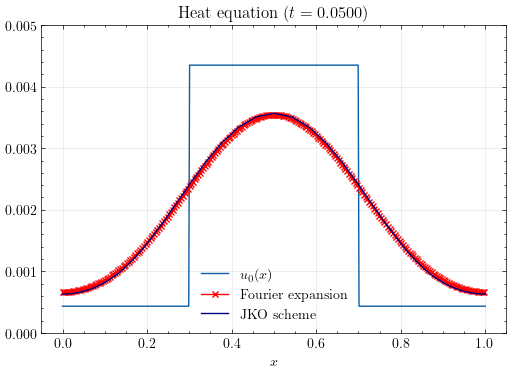}
  \end{subfigure}
  \begin{subfigure}{.24\textwidth}
    \centering
    \includegraphics[width=\linewidth]{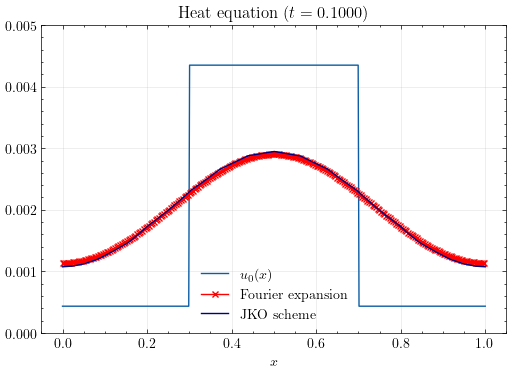}
  \end{subfigure}
  \begin{subfigure}{.24\textwidth}
    \centering
    \includegraphics[width=\linewidth]{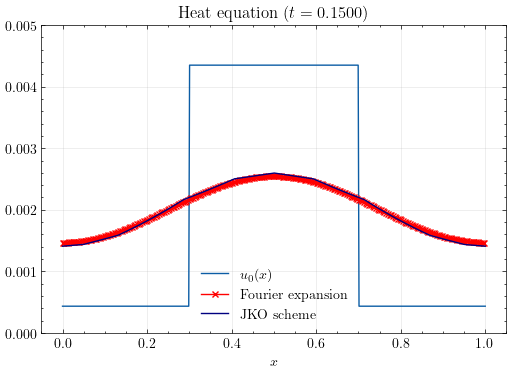}
  \end{subfigure}
  \caption{Simulation of the heat equation using unbalanced optimal transport.}
  \label{fig:heat-eq}
\end{figure}

As a final illustration, we show how our method can be used in computing the 
JKO scheme for Wasserstein gradient flows~\cite{santambrogio2025optimal}. To 
this end, we propose a toy experiment consisting in computing the JKO scheme 
on a grid $(x_\ell:= x_0 + \ell h)_{\ell\in\IntInterval{1}{n}}$ where $x_0\in\R$ and 
$h>0$. This amounts to computing the following iterates 
\begin{equation}
\label{eq:jko}
    \bm{u}_{k+1} = (\tpplan^k)^\top\indic
    \quad
    \mathrm{where}
    \quad 
    \tpplan^k=\underset{\tpplan\geq 0,\tpplan\indic =\bm{u}^k}{\mathrm{argmin}} 
    ~\tpplan \!:\! \costmat + \frac{1}{\tau}\KL(\tpplan^\top \indic,\indic/n)
\end{equation}
using $\costmat_{ij} := (x_i-x_j)^2$ and starting from $u_0 \in \R$. Under 
suitable conditions on $h/\tau$, the iterates~\eqref{eq:jko} converge to the 
solution of the heat equation $\partial_t u = \Delta u$ with initial condition 
$u(0,\cdot) = u_0$ and no-flux condition at the 
boundaries~\cite{hraivoronska2025convergence}). The transport plans computed 
in~\eqref{eq:jko} are obtained by solving a semi-unbalanced optimal transport 
with $\KL$ relaxation. Figure~\ref{fig:heat-eq} presents the solutions 
obtained using Algorithm~\ref{alg:adapted-ns}, and shows that our approach 
indeed approximates the JKO scheme.

\section{Conclusion}
\label{sec:conc}

In this work, we proposed a framework inspired by the network simplex method 
for (discrete) unbalanced and semi-unbalanced optimal transport. Our method is 
storage-efficient compared to existing schemes, while being amenable to a 
complexity analysis. Moreover, it can be implemented in closed form for 
quadratic and Kullback-Leibler variants. Our numerical results show that the 
adapted network simplex algorithm is competitive with mature solvers on 
several tasks, including problems that rely on unbalanced optimal transport 
as a subroutine.

The focus of our work was on obtaining a simple, practical version of the 
network simplex. Still, exploring variants of our method with improved 
complexity guarantees is a possible continuation of this paper. An extensive 
study of our method's potential for Wasserstein barycenter problems, from 
both a theoretical and a numerical perspective, is also a potential axis for 
future research.

\paragraph{AI statement}

Claude Opus 5.1-5 was used to check the proofs and lead to slightly improved 
complexity results. It was also used to generate the \cpp~code from the 
\python~implementation, and to generate some 
standard graph routines. The authors retain full responsibility for the content 
of the manuscript and the code.

%

\bibliographystyle{plain}
\bibliography{biblio}
\end{document}